\documentclass[11pt]{amsart}

\usepackage{enumerate}
\usepackage{lscape}
\usepackage{amsopn}
\usepackage{amsmath}
\usepackage{mathrsfs}
\usepackage{algorithm}
\usepackage{algpseudocode}
\usepackage{algorithmicx}
\usepackage{url}
\usepackage{amssymb}
\usepackage{graphicx}

\DeclareMathOperator{\support}{supp}
\DeclareMathOperator{\cov}{cov}
\DeclareMathOperator{\diag}{diag}
\DeclareMathOperator{\ddim}{dim}
\DeclareMathOperator{\Span}{span}
\DeclareMathOperator{\Sq}{\boldsymbol{2}}
\DeclareMathOperator{\rank}{rank}
\DeclareMathOperator{\Dom}{Dom}
\DeclareMathOperator{\argmin}{argmin}
\DeclareMathOperator{\Ker}{Ker}
\DeclareMathOperator{\Hess}{Hess}
\DeclareMathOperator{\trace}{Trace}
\DeclareMathOperator{\Card}{Card}

\DeclareMathOperator{\sign}{sign}
\newcommand{\reals}{\mathbb{R}}
\newcommand{\realspos}{\reals_{\geq 0}}
\newcommand{\naturals}{\mathbb{N}}
\newcommand{\Sqm}[1]{(#1)^{\Sq}}

\theoremstyle{theorem}%
\newtheorem{theorem}{Theorem}
\newtheorem{proposition}{Proposition}

\theoremstyle{defintion}%
\newtheorem{experiment}{Experiment}%
\newtheorem{remark}{Remark}%
\newtheorem{setass}{Set of Assumptions}
\newtheorem{problem}{Problem}

\author{Federico Piazzon}
\address{Department of Mathematics "Tullio Levi-Civita" University of Padua}
\email{fpiazzon@math.unipd.it}
\title[Computing Optimal Designs]{Computing Optimal Experimental Designs on Finite Sets by Log-determinant Gradient Flow}

\begin{document}

\begin{abstract}
Optimal experimental designs are probability measures with finite support enjoying an optimality property for the computation of least squares estimators. We present an algorithm for computing optimal designs on finite sets based on the long-time asymptotics of the gradient flow of the log-determinant of the so called information matrix. We prove the convergence of the proposed algorithm, and provide a sharp estimate on the rate its convergence. Numerical experiments are performed on few test cases using the new matlab package \textsc{OptimalDesignComputation}.
\end{abstract}

\keywords{optimal design, gradient flow, least squares, polynomial fitting}



\maketitle

\section{Introduction}\label{secIntroduction}
\subsection{Optimal experimental designs and optimal measures}
Let $X\subset\reals^n$ be a compact set and $\Phi:=\{\phi_1,\phi_2,\dots,\phi_N\}$ be a set of bounded real functions on $X$ which are linear independent on $X.$ A statistical model based on $\Phi$ is the linear combination
\begin{equation}\label{model}
\phi(x):=\sum_{j=1}^N c_j\phi_j(x),
\end{equation}
where the coefficients $c_j$'s are unknown. Typically one aims at reconstructing $\phi$ starting from some noisy measuraments of it, e.g., the observations 
$$\tilde \phi_i:=\phi(x)+\epsilon_i,\;\;i=1,2,\dots M\geq N,$$
where the $\epsilon_i$'s are i.i.d. Gaussian variables with zero mean and variance $\sigma^2.$ In such a situation least squares is the standard tool. Indeed one can estimate the parameter $c\in \reals^N$ by
$$\hat c:=(V^tV)^{-1}V^t \tilde\phi,$$
where 
$$V_{i,j}:=\left(V(X,\Phi)\right)_{i,j}=\phi_j(x_i),\;\;i=1,2,\dots,M,\;j=1,2,\dots,N$$
is the \emph{Vandermonde matrix} of the basis $\Phi$ at the points $\{x_1,x_2,\dots,x_M\}\subseteq X.$ Using the properties of Gaussian random variables, it is not hard to show that 
$$\cov(\hat c)=\sigma^2(V^tV)^{-1}.$$
\emph{We warn the reader that we will often prefer the compact notation $V$ instead of $V(X,\Phi)$ when $X$ and $\Phi$ are clarified from the context.}

Recall that, roughly speaking, the \emph{generalized variance} $\det\cov((\cdot))$ of a random variable is a measure of how much its density spreads around its mean. Therefore, if  we choose $\{x_1,x_2,\dots,x_M\}$ maximizing $\det(V^tV)$, then we obtain an estimator $\hat c$ of $c$ which is optimal in terms of the concentration of its probability density.

The above construction can be generalized to the case of weighted least squares. In such a case we consider finite \emph{designs}, e.g., $((x_1,x_2,\dots,x_M),(w_1,w_2,\dots,w_M))\in X\times \reals_{\geq 0}^M$, instead of just $M$ points, and we estimate the parameter $c$ by the weighted least squares estimator with nodes $(x_1,x_2,\dots,x_M)$ and weights $(w_1,w_2,\dots,w_M)$, i.e. ,
$$\hat c_w:=(V^tWV)^{-1}V^tW \tilde\phi,$$
where $W:=\diag(w)$ and $V$ is as above. The generalized variance of $\hat c_w$ is 
\begin{equation}
\det(\cov(\hat c_w))=\det((V^tWV)^{-1})=\frac 1{\det(V^tWV)}=:\frac 1 {\det(G(w;X,\Phi))}.
\end{equation}
Note that the Gram matrix $G(w;X,\Phi)$ is custumary called \emph{information matrix} in the context of statistics. Recall also that the Gram matrix $G(w;X,\Phi)$ is often written in statistical context in the form
$$G(w;X,\Phi)=\sum_{i=1}^M w_i V_{i,:}^tV_{i,:},$$
where, for any $i=1,2,\dots, M$, $V_{i,:}$ is the row vector $(V_{i,1},\dots,V_{i,N}).$ 

The above heuristics motivates the following definition. A \emph{D-optimal experimental design} for the model \eqref{model} is a design $\mu=(x,w)\in X^{M'}\times \reals_{\geq 0}^{M'}$ with $M'\leq M$, having maximal determinant of the generalized information matrix among all finite designs $\nu=(y,v)$ having mass $1$, i.e., $\|w\|_1=\|v\|_1=1.$ 
\begin{remark}
Note that, in our finite design space setting, given a D-optimal design as defined above, we can always identify it with the design $(\tilde x,\tilde w)$, with $\tilde x=X$ and  $\tilde w\in \realspos^M$ which extends $w$ to $0.$ For this reason we will work only on weights $w\in \realspos^M$, and term \emph{D-optimal design} a vector $w^*$ of non-negative weights (possibly with vanishing components) with $\|w^*\|_1=1$, having maximal determinant of the generalized information matrix among all such vectors. 
\end{remark}

D-optimal experimental designs are a particular instance of the so-called \emph{optimal measures} studied in approximation theory and pluripotential theory, in the more general case of $X$ being an infinite compact set, see \cite{BlBoLeWa10}. Indeed, the only difference between the two mathematical objects is that optimal measures are not required to have finite support. Namely, given $X$ and $\Phi$ as above, a probability measure $\mu\in \mathcal M_1(X)$ is termed an \emph{optimal measure} if 
\begin{align}
\det G(\mu;\Phi,X)&:=\det\left(\int_X \phi_i\phi_j d\mu  \right)_{i,j=1,2,\dots,N}\notag\\
&\geq \det G(\nu;\Phi,X),\;\;\forall \nu\in \mathcal M_1(X).\label{optimalmeas}
\end{align}   
It is worth pointing out that, by Tchakaloff Theorem \cite{Tc57} (see also \cite{Pu97,BaTe06}), any probability measure $\mu\in \mathcal M_1(X)$ admits a positive quadrature rules which is exact on $\Phi^2:=\Span\{\phi_i\cdot\phi_j:\,i,j=1,2,\dots,N\}$, supported in $X$, and having at most $\ddim_X\Phi^2$ quadrature nodes. Therefore, if $\mu\in \mathcal M_1(X)$ is an optimal measure, then there exists an optimal design with cardinality at most $M\leq\ddim_X\Phi^2.$ On the other hand, any design $(x,w)$ on $X$ can be canonically identified with a probability measure $\mu\in \mathcal M_1(X)$ by setting $\mu:=\sum_{i=1}^M w_i \delta_{x_i}.$ 

There is also a lower bound for the cardinality of the support of a D-optimal design. Indeed, assuming as above that $V(X;\Phi)$ has full rank (i.e., functions $\phi_i$, with $i=1,2,\dots,N$, are linear independent on $X$) then there exists at least a design having non zero $\det G(w;X,\Phi)$, thus, for any D-optimal design $w^*$, we have $\det G(w;X,\Phi)\neq 0$. This in particular implies that, denoting by $I=(i_1,i_2,\dots i_{M'})$ the indices for which $w_{i_k}^*>0$,  $diag(\sqrt{w_{i_1}},\dots,\sqrt{w_{i_{M'}}})V_{I,:}$ has full rank, hence we have $\Card I\geq N.$ 

It is worth stressing that the above definition of optimal measures is equivalent to another optimality property. Let $\{\phi_1(\cdot;\mu),\phi_2(\cdot;\mu),\dots,\phi_N(\cdot;\mu)\}$ be the $L^2_\mu(X)$-orthonormal basis of $\Phi$ computed by Gram-Schmidt algorithm starting from $\{\phi_1,\phi_2,\dots,\phi_N\}$. Then we can define the reproducing kernel $K(\cdot,\cdot;\mu)$ and the Bergman function $B(\cdot;\mu)$ of the space $\Span \Phi$ endowed by the $L^2_\mu(X)$ scalar product by setting
\begin{align*}
K(x,y;\mu):=\sum_{i=1}^N \phi_i(x;\mu)\phi_i(y;\mu),\;\;\;\;B(x;\mu):=K(x,x;\mu).
\end{align*}
Any optimal measure $\mu^*$ enjoy the property that 
$$\max_{x\in X}B_{\mu^*}(x)=\min_{\mu\in \mathcal M_1(X)}\max_{x\in X}B_\mu(x)=N,$$
which is indeed equivalent to \eqref{optimalmeas} above, see \cite[Prop. 3.1]{BlBoLeWa10}. Note also an analog property holds for designs, being this result a part of the Kiefer Wolfowitz Theorem. Precisely, we denote by $\{\phi_1(\cdot;w),\phi_2(\cdot;w),\dots,\phi_N(\cdot;w)\}$ be the orthonormal basis of $\Phi$ with respect to the scalar product
$$\langle\phi;\psi\rangle_w:=\sum_{i=1}^M\phi(x_i)\psi(x_i)w_i$$
computed by Gram-Schmidt algorithm starting from $\{\phi_1,\phi_2,\dots,\phi_N\}$. Then, using the notation
\begin{align}
& K(x,y;w):=\sum_{i=1}^N \phi_i(x;w)\phi_i(y;w),\label{reproducingkernel}\\
& B(x;w):=K(x,x;w),\label{bergmanfunction}
\end{align}
any D-optimal desig $w^*$ is also \emph{G-optimal}, i.e.,
$$\max_{x\in X}B(x;w^*)=\min_{\|w\|_1=1,w_i\geq 0} \max_{x\in X}B(x;w)=N.$$

The literature concerning the study of optimal designs is so ample that we can not even summarize it, we refer the reader to \cite{Pu93} and references therein for an extensive treatment of the subject. Even though the optimal designs are a relatively old topic which goes back to the work of Kiefer and Wolfowitz \cite{KiWo59,Ki61}, the research in this area is still very active, expecially concerning the computation of exact optimal designs or the study of efficient algorithms for their approximation, see e.g., \cite{LuPo13}, \cite{HeNa18} and references therein. The first and perhaps most famus algorithm for the computation of D-optimal design is due to Titterington \cite{SiTiTo78}, we refer the reader also to \cite{Yu10} where the Titterigton algorithm (also known as \emph{multiplicative algorithm}) is studied in a wider framework of a class of algorithms.    

D-optimal designs are in general computed, or approximated, numerically by iterative algorithms. In this context there are two opposite situations. If the set $X$ is finite, then the iterative optimization algorithm can generally run on the whole set $X$. In contrast, if $X$ is an infinite set (or a continuum), a good finite representer, say $\tilde X$, of $X$ must be constructed first. There are several approaches to attack such a problem, treating $\tilde X$ as a variable or a fixed parameter of the problem, which is chosen accordingly to certain heuristics as \emph{space filling} techniques, grid exploration (see e.g., \cite{HaFiRo21}), or \emph{minimal spanning tree}. Recently, it has been shown that the use of \emph{polynomial admissible meshes} gives precise quantitative estimates of the approximation intruduced in the discretization of the problem , i.e., when passing from $X$ to $\tilde X$, see \cite{BoPiVi20}.

A slightly different approach to the problem of computing optimal design for infinite semialgebraic sets has been recently proposed in \cite{DeGaHeHeLa19} where a sequence of relaxed problems is considered. 

\subsection{Our contribution}
In the present work we consider the problem of computing D-optimal designs for the model \eqref{model} and the \emph{finite} set $X$ assuming that the space of functions $\Phi:=\{\phi_1,\dots,\phi_N\}$ has dimension $N$ on $X$.

Our strategy relies on three steps. First, in Subsection \ref{subsec:equivalence}, we re-formulate the problem of computing D-optimal designs via a sequence of equivalent optimization problems that leads to deal with an \emph{unconstrained optimization} of a \emph{real analytic objective} $F$, where
$$F(z)=-\frac 1 N \log \det G((z_1^2,z_2^2,\dots,z_M^2);X,\Phi)+\|(z_1^2,z_2^2,\dots,z_M^2)\|^2.$$
We prove the equivalence of all the considered problems in Theorem \ref{thmequivalenve}.

In Subection \ref{subsecWellPosedness} we study conditions ensuring the well-posedness of the problems that we introduced. Then, we show in Subsection \ref{subsecIllPosed} how the ill-posed case (presence of non-unique optimal designs) can be easily regularized by introducing a different problem admitting an unique solution which is a distinguished optimal design. We are able to prove in Subsection \ref{sub:error} sharp \emph{a posteriori} error bounds both for the well-posed case and the regularization of the ill-posed case.

As second step, we consider the \emph{gradient flow} of $F$, i.e., 
\begin{equation*}
z'=-\nabla F(z),
\end{equation*}
 as a tool for minimization. Indeed we can prove in Theorem \ref{thm:Gf} that such equation has a unique solution for any choice of the initial data $z^0$. The trajectories of the flow are real analytic both in time and in the initial data $z^0$. Moreover all the trajectories converge to a minimizer $z^*$ of $F$ as $t\to +\infty$: a D-optimal design $w^*$ then is easily recovered. 

Lastly, we design a suitable infinite horizon numerical integration scheme for such a flow. We define Algorithm \ref{algo1} by combining the \emph{Backward Euler Scheme} with the Newton's Method. The convergence of Algorithm \ref{algo1} is proven in Theorem \ref{thm:convergence} under the hypothesis of a suitable choice of the time step. This convergence result is obtained showing first the consistency of the algorithm (Theorem \ref{thm:concistency}) and combining it with a stability estimate (Proposition \ref{prop:stability}). Since it is not possible to compute the upper bound for the time step which is mandatory for applying Theorem \ref{thm:convergence}, we modify the proposed algorithm including an adaptive choice of the time step. This results in Algorithm \ref{algo2}.

Finally we present in Section \ref{SecExperiment} some numerical experiments for testing the performances of the proposed algorithms. 
\section{Design optimization problems}
\subsection{Equivalence of three optimization problems}\label{subsec:equivalence}
The computation of a D-optimal design can be formulated as the solution of the optimization problem
\begin{problem}\label{P1}
Find $w^*\in \realspos^M$ such that
\begin{equation}
\det G(w^*;\Phi,X)=\max\{\det G(w;\Phi,X):\,w\in\realspos^M, \|z\|_1=1\}.
\end{equation}
\end{problem}
We prefer to remove the mass-constraint by using Lagrange multipliers and consider a minimization instead of maximization. So we introduce the \emph{design energy}
\begin{equation}
E(w):= -\frac 1 N\log\det G(w;\Phi,X)+\|w\|_1
\end{equation}
and consider:
\begin{problem}\label{P2}
Find $w^*\in \realspos^M$ such that
\begin{equation}
E(w^*)=\min\{E(w):\,w\in\realspos^M\}.
\end{equation}
\end{problem}
In order to remove the positivity constraint appearing in Problem \ref{P1} we define the \emph{coordinate square map} $\Sqm{\cdot}:\reals^M\rightarrow \realspos^M$, $\Sqm{z}:=(z_1^2,z_2^2,\dots,z_M^2)^t$, introduce
\begin{equation}
F(z):=E(\Sqm{z})= -\frac 1 N\log\det G(\Sqm{z};\Phi,X)+\|z\|_2^2,
\end{equation}
and consider
\begin{problem}\label{P3}
Find $z^*\in \reals^M$ such that
\begin{equation}
F(z^*)=\min\{F(z):\,z\in\reals^M\}.
\end{equation}
\end{problem}
\begin{remark}
We stress that passing from Problem \ref{P2} to Problem \ref{P3} we loose the relevant property of the convexity of the objective functional. On the other hand, we remove the inequality constraint of Problem \ref{P2}. In Subsection \ref{subs:convergenceanalysis} it will be clarified that the convexity of $E$ is still playing an important role, even when dealing with $F$.
\end{remark}

All the considered problems are indeed equivalent, more precisely we have the following result.
\begin{theorem}[Equivalence of problems \ref{P1}, \ref{P2}, and \ref{P3}]\label{thmequivalenve}
Let $X$ and $\Phi$ be as above. Then 
\begin{enumerate}[i)]
\item $w^*\in \realspos^M$ solves Problem \ref{P1} if and only if it solves Problem \ref{P2},
\item $z^*\in \reals^M$ solves Problem \ref{P3} if and only if $w^*:=\Sqm{z^*}$ solves Problem \ref{P1}.
\end{enumerate}
In any of the above cases, $(\support{w^*},w^*)$ is a D-optimal design for $\Phi$ on $X$.
\end{theorem}
Before proving Theorem \ref{thmequivalenve}, it is convenient to state and prove some analytical and geometrical properties of $E$ and $F$.
\begin{proposition}
The function $E$ is \emph{convex} and \emph{real analytic} on $\{w\in\realspos^M: \rank V(\support w,\Phi)=N\}.$ The function $F$ is real analytic on its domain. Moreover we have
\begin{align}
\partial_i E(w)&=1-\frac{B(x_i;w)}{N}\label{nablaE}\\
\partial^2_{i,j}E(w)&=\frac{K(x_i,x_j;w)^2}N\label{hessianE}\\
\partial_i F(z)&=2z_i\left( 1-\frac{B(x_i;\Sqm{z})}{N} \right)\label{nablaF}\\
\partial^2_{i,j}F(z)&=4z_iz_j \frac{K(x_i,x_j;\Sqm{z})^2}N+2\delta_{i,j}\left( 1-\frac{B(x_i;\Sqm{z})}{N} \right)\label{hessianF}
\end{align}
\end{proposition}
\begin{proof}
Convexity of $E$ and real analyticity of $E$ and $F$ are consequences of the properties of the elementary fuctions used in their definitions. 

For notational convenience we denote $G(w,\Phi,X)$ by $G(w)$, being $X, \Phi$ fixed here. Let us compute the first derivatives of $E$. Using the Jacobi formula, we can write
\begin{align*}
&\partial_i E(w)=1-\frac 1 N \frac{\trace(G^*(w)\partial_i G(w))}{\det G(w)}=1- \frac{\trace(G^{-1}(w)\partial_i G(w))}{N}\\
=&1- \frac{\trace(G^{-1}(w)V_{i,:}^tV_{i,:})}{N}=1- \frac{\sum_{j=1}^N(G^{-1}(w)V_{i,:}^tV_{i,:})_{j,j}}{N}\\
=&1- \frac{\sum_{j=1}^N\sum_{h=1}^M G^{-1}_{j,h}(w)V_{i,h}V_{i,j})_{j,j}}{N}=1-\frac{V_{i,:}G^{-1}(w)V_{i,:}^t} N.
\end{align*}
In other words
\begin{equation}
\nabla E(w)=1-\frac 1 N \diag( V G^{-1}(w) V^t). \label{derivativefirstformula}
\end{equation}
Let us write $G(w)=A(w)A^t(w)$, where $A_{i,j}(w)=\langle \phi_i(\cdot);\phi_j(\cdot;w)\rangle_w$ is the matrix representing the change of basis diagonalizing $G(w)$. We have $G^{-1}(w)=A^{-t}(w)A^{-1}(w)$. Hence we have
\begin{equation*}
\nabla E(w)=1-\frac 1 N \diag( V A^{-t}(w)A^{-1}(w) V^t). 
\end{equation*} 
Note that $A(w)\tilde V^t(w)=V^t$, where $\tilde V(w)_{i,j}=\phi_j(x_i;w)$ is the Vandermonde matrix of the $w$-orthonormal basis $\{\phi_j(\cdot;w)\}_{j=1,\cdot, N}.$ Therefore we have
\begin{equation*}
\nabla E(w)=1-\frac 1 N \diag( \tilde V(w)\tilde V(w)^t)=\left(1-\frac{\sum_{j=1}^N\phi_j(x_i;w)^2}N\right)_{i=1,2,\dots,M},
\end{equation*}
i.e., \eqref{nablaE} holds true.

Let us compute the second order derivatives.
\begin{align*}
&\partial^2_{i,j}E(w)= \partial_j\left( 1-\frac{V_{i,:}G^{-1}(w)V_{i,:}^t}N\right)=-\frac{V_{i,:}\partial_j G^{-1}(w)V_{i,:}^t}N\\
=&\frac{V_{i,:}G^{-1}(w)\partial_j G(w)G^{-1}(w)V_{i,:}^t}N=\frac{V_{i,:}G^{-1}(w)V_{j,:}^tV_{j,:}G^{-1}(w)V_{i,:}^t}N\\
=&\frac{\left( V_{i,:}G^{-1}(w)V_{j,:}^t \right)^2}N=\frac{\left( \tilde V_{i,:}(w)\tilde V_{j,:}^t(w) \right)^2}N=\frac{\left(\sum_{h=1}^N\phi_h(x_i;w)\phi_h(x_j;w)\right)^2}N\\
=&\frac{K(x_i,x_j;w)^2}N.
\end{align*}
This concludes the proof of \eqref{hessianE}. Note that \eqref{nablaF} and \eqref{hessianF} easily follows from \eqref{nablaE} and \eqref{hessianE}, respectively.
\end{proof}
We are now ready to prove Theorem \ref{thmequivalenve}.
\begin{proof}[Proof of Theorem \ref{thmequivalenve}]
First note that maximizing the function $ $ or minimizing its composition with the strictly decreasing function $-\log(\cdot)$ are equivalent problems. Let us notice that the function $E$ is convex, thus $\mathcal E:=E\llcorner_{\{w:\sum w_i=1,w_i\geq 0\}}\,-1$ is convex as well. Indeed, for any $w\in \realspos^M$ and $u\in \reals^M$, using \eqref{hessianF} and \eqref{bergmanfunction} we obtain 
\begin{align}
&Nu^t\Hess E(w)u\notag\\
=&\sum_{i=1}^M\sum_{j=1}^Mu_iu_j\left(\sum_{h=1}^N \phi_h(x_i;w)\phi_h(x_j;w)\right)^2\notag\\
=&\sum_{i=1}^M\sum_{j=1}^Mu_iu_j\sum_{k=1}^N\sum_{h=1}^N \phi_h(x_i;w)\phi_h(x_j;w)\phi_k(x_i;w)\phi_k(x_j;w)\notag\\
=&\sum_{k=1}^N\sum_{h=1}^N \sum_{i=1}^M\sum_{j=1}^Mu_iu_j\phi_h(x_i;w)\phi_h(x_j;w)\phi_k(x_i;w)\phi_k(x_j;w)\notag\\
=&\sum_{h=1}^N\sum_{k=1}^N\left(\sum_{i=1}^M u_i\phi_h(x_i;w)\phi_k(x_i;w)\right)^2\geq 0. \label{hessianformula}
\end{align}
Assume $w^*$ is a solution of Problem \ref{P1}. Then it is a minimizer of $\mathcal E$ on $\{w:\sum w_i=1,w_i\geq 0\}$. The function $\mathcal E$ is convex, thus the Karush Kuhn Tucker sufficient conditions for minimizers are also necessary. That is, $w^*$ satisfy the following. There exists $\lambda\in \reals$ and $\theta\in \realspos^M$ such that
\begin{equation}\label{kktP1}
\begin{cases}
\nabla \mathcal E(w^*)+\lambda (1,1,\dots,1)^t+\theta=0\\
\sum_{i=1}^M w_i^*=1\\
w_i^*\geq 0 \text{ for all }i\\
w_i^*\theta_i=0 \text{ for all }i
\end{cases}\;.
\end{equation}
Scalar multiplying by $w^*$ the first equation, using the remaining equations, equation \eqref{nablaE}, and \eqref{bergmanfunction}, we obtain
$$\lambda=N^{-1}\sum_{i=1}^M B(x_i;w^*)w_i^*=N^{-1}\sum_{j=1}^N\sum_{i=1}^M \phi_j^2(x_i;w^*)w_i^*=1.$$
Therefore $w^*$ satisfies also
\begin{equation}\label{kktP2}
\begin{cases}
\nabla \mathcal E(w^*)+ (1,1,\dots,1)^t+\theta=0\\
w_i^*\geq 0 \text{ for all }i\\
w_i^*\theta_i=0 \text{ for all }i
\end{cases}\;\equiv
\begin{cases}
\nabla E(w^*)+\theta=0\\
w_i^*\geq 0& \text{ for all }i\\
w_i^*\theta_i=0& \text{ for all }i\\
\theta_i\leq 0& \text{ for all }i
\end{cases}\;.
\end{equation}
But this last set of equations corresponds to the Karush Kuhn Tucker sufficient conditions for the minimization of $E$ on $\realspos^M.$ Conversely, if $w^*$ is a solution of Problem \ref{P2}, then (again by convexity) it needs to satisfy the right hand side system of equations in \eqref{kktP2} and hence the left one. If we show that $w^*$ has mass $1$ (i.e., $\sum_{i=1}^M w_i^*$), then $w^*$ satisfies \eqref{kktP1} as well and hence it solves Problem \ref{P1}. For, we scalar multiply the first equation of the right bock of \eqref{kktP2} and obtain $\sum_{i=1^M}\partial_i E(w^*)w_i^*=0$ that leads to 
$$\sum_{i=1}^M w_i^*=N^{-1}\sum_{i=1}^MB(x_i;w^*)w_i^*=1.$$
This concludes the proof of the equivalence of Problem \ref{P1} and Problem \ref{P2}. 

The equivalence of Problem \ref{P2} with Problem \ref{P3} immediately follows from the topological properties of the coordinate square map $z\mapsto \Sqm{z}$. Thus Problem \ref{P3} is also equivalent to Problem \ref{P1}.
\end{proof}

\subsection{Well-posedness of problems \ref{P2} and \ref{P3}}\label{subsecWellPosedness}
Existence of solutions of problems \ref{P2} and \ref{P3} is straightforward and can be obtainded through the direct method. We need to investigate sufficient conditions for the uniqueness of solutions. In the case of Problem \ref{P3} the word uniqueness need to be clarified since, due to the symmetry of $F$, is not possible to obtain a unique solution $z^*$ of Problem \ref{P3} unless $z^*=0$. Indeed we need to look for sufficient conditions for the uniqueness of $\Sqm{z^*}$.

As first attempt, one may require $E$ to have positive definite Hessian at any point. This is in general a too restrictive assumption, both in view of Tchakaloff Theorem and of the following characterization of the kernel of the Hessian of $E$.
\begin{proposition}\label{propKernel}
For any $w\in \realspos^M$ and $u\in \reals^M$ we have
\begin{equation}
u^t\Hess E(w)u=0\;\;\Leftrightarrow\;\;V(\Phi^2;X)^tu=0.
\end{equation}
\end{proposition}
\begin{proof}
We already shown (see \eqref{hessianformula}) that, for any $w\in \realspos^M$ and $u\in \reals^M$, $u^t\Hess E(w)u=\sum_{h=1}^N\sum_{k=1}^N\left(\sum_{i=1}^M u_i\phi_h(x_i;w)\phi_k(x_i;w)\right)^2.$
This expression vanishes if and only if $\sum_{i=1}^M u_i\phi_h(x_i;w)\phi_k(x_i;w)=0$ for all $h,k=1,2,\dots, N$, i.e., if $V(\Phi^2;X)^tu=0$.
\end{proof}

\begin{remark}
The above proposition shows in particular that $E$ is strongly convex if and only if $\Ker V(\Phi^2;X)^t=0.$ That is, if and only if $\ddim_X \Phi^2=\Card X=M.$ This is clearly too restrictive for many applications in which, for instance, $X$ is a discretization of a infinite set, or in which $M>>N.$ 
\end{remark}
We consider the following less restrictive hypothesis.
\begin{setass}\label{H2}
There exists $w^*\in \argmin_{\realspos^M} E$ such that
\begin{equation}\label{transversality}
\left(\{w^*\}+\Ker \Hess E(w^*)\right)\cap \realspos^M=\{w^*\}.
\end{equation} 
\end{setass}
Not only uniqueness of $w^*$ easily follows from \eqref{transversality}, it is indeed equivalent as we state and prove in the next proposition.
\begin{proposition}[Well-posedness of Problem \ref{P2} under \eqref{transversality}]
The property \eqref{transversality} holds if and only if  $w^*$ is the unique solution of Problem \ref{P2}. In such a case Problem \ref{P3} admits exaxtly $2^M$ solutions $z^{*,j}$ for which $\Sqm{z^{*,j}}\equiv w^*$ for all $j=1,2,\dots,2^M.$
\end{proposition}
\begin{proof}
The proof can be carried out by contradiction. Assume there exist $w^*,v^*\in \argmin E.$ From convexity of $E$ it follows that E is constant along the whole segment $[w^*,v^*].$ Therefore, setting $u:=w^*-v^*$ and $w\in ]w^*,v^*[$, we have $u^t\Hess E(w)u=0$. Due to Proposition \ref{propKernel}, we have $V(\Phi^2;X)^tu=0$, which implies $w^*\neq w\in \left(\{w^*\}+\Ker \Hess E(w^*)\right)\cap \realspos^M.$ Hence \eqref{transversality} implies uniqueness of $w^*.$

Conversely, if we assume that \eqref{transversality} does not hold, by Proposition \ref{propKernel} we can find $[w^*,v^*]\subset\realspos^M$ with $u:=w^*-v^*\in \Ker V(\Phi^2;X)^t$. Note that any $w\in ]w^*,v^*[$ has the same moments on $\Phi^2$ as $w^*$ (and $v^*$), i.e.,
\begin{equation}
\sum_{i=1}^M \phi(x_i)w^*_i=\sum_{i=1}^M \phi(x_i)w_i,\;\;\forall \phi\in \Phi^2.
\end{equation} 
This in particular implies equality of Gram matrices and hence equality of $E(w^*)$, $E(v^*)$ and $E(w),$ i.e., Problem \ref{P2} does not admit a unique solution.
\end{proof}

\subsection{Regularizing the ill-posed case}\label{subsecIllPosed}
The quest for a regularization of the ill-posed case (i.e., when \eqref{transversality} does not hold) naturally arises. Indeed, due to the particular geometrical features of Problem \ref{P2}, it is possible to define a regularized problem having a unique solution which is also a solution of Problem \ref{P2}. The idea is rather easy: if \eqref{transversality} does not hold, then the solution set $S$ of Problem \ref{P2} is the subset of vectors with non-negative components lying in the affine variety $\mathcal A:=\{x\in \reals^M:\;V(\Phi^2;X)^tw=m\}$, where $m$ is the vector of the moments of any optimal design on a basis of $\Phi^2_X.$ Note that $S$ is necessarily compact, being $E$ coercive, e.g., if $\|w\|_1\to+\infty$ then $E(w)\to +\infty.$

We can consider, among all solutions $w^*$ of Problem \ref{P2}, the one, say $\hat w^*$, with the smallest squared norm of the projection $\pi_{K}$ on $K:=\Ker V(\Phi^2;X)^t,$ i.e.,
\begin{equation}\label{minimalnormsolution}
\|\pi_K\hat w^*\|_2^2=\min\left\{\|\pi_K w^*\|_2^2: w^*\in \argmin_{\realspos^M} E\right\}.
\end{equation}
Now assume that 
\begin{equation}\label{orthmagic}
\exists\,\hat w^*\in \argmin_{\realspos^M} E\,\text{ such that }\pi_K\hat w^*=0.
\end{equation}
 Then $\hat w^*$ is a minimizer of $E$ on $\realspos^M$ and a minimizer of $\|\pi_K\cdot\|_2^2$. Thus $\hat w^*$ is also a global minimizer of $E(w)+\|\pi_K w\|^2.$

If we try to reason the other way around, namely, we look for a (possibly unique) minimizer  of $E(w)+\|\pi_K w\|^2$ on $\realspos^M$ we need to distinguish two situations. Precisely, either 
\begin{equation}\label{badcase}
S\cap K^\perp \text{ is the  empty set,}
\end{equation} 
or
\begin{equation}
S\cap K^\perp=\{\bar w\}.
\end{equation} 
In this last case, if we solve the optimization problem for $E(\cdot)+\|\pi_K\cdot\|^2$, then we end up with a solution of Problem \ref{P2}. Instead, in the case of \eqref{badcase} a minimizer of $E(\cdot)+\|\pi_K\cdot\|^2$ does not need to be a minimizer of $E$. In order to overcome such a difficulty we decide to introduce and study a family of slightly modified functions, namely, for any $\eta>0$, we set
\begin{equation}
E_\eta(w):=E(w)+\eta \|\pi_K w\|_2^2.
\end{equation}

Indeed, minimizing $E_\eta$ for a sequence of values of $\eta$ tending to $0$, we  construct a sequence of approximations of a solution of Problem \ref{P2}.
\begin{proposition}
Let $\{\eta_n\}\downarrow 0$. For any $n\in \naturals$ there exists a unique $w_n^*$ minimizing $E_{\eta_n}$ on $\realspos^M$. The sequence $\{w_n^*\}$ converges to the optimal design $\hat w^*$ defined in \eqref{minimalnormsolution}.
\end{proposition}
\begin{proof}
First notice that for any $n\in \naturals$ the function $E_{\eta_n}$ is strongly convex and thus any minimizer of $E_{\eta_n}$ on $\realspos^M$ is necessarily unique. For, let $v\in \reals^M$. By direct computation we can show that
\begin{equation}
v^t \Hess E_{\eta_n}(w)v=v^t\Hess E(w)v+\eta_n\|\pi_K v\|^2.
\end{equation}
So $v^t \Hess E_{\eta_n}(w)v=0$ implies $v\in \Ker \Hess E_{\eta_n}(w)=K$ and $\|\pi_K v\|=0$, so $v\in K\cap K^\perp$, i.e., $v=0.$

Now let us pick $w_n^*$ as in the statement. We clearly have $\sup\|w_n^*\|_1<\infty$ since $E$ is coercive and thus $E_{\eta_n}$ is. Let us extract a converging subsequence and relabel it. Also denote by $\bar w$ its limit. We need to show that $\bar w$ is a minimizer of $E.$ Note that, for any $w^*\in \argmin E$ we have
\begin{equation}
E(w^*)+\eta_n \|\pi_K w^*\|^2\geq E(w_n)+\eta_n \|\pi_K w_n\|^2\geq E(w^*).
\end{equation}
Since $E$ is continuous, letting $n\to +\infty$ we obtain $E(\bar w)=E(w^*)=\min_{\realspos^M} E.$

In order to conclue the proof we are left to show that $\hat w^*=\bar w.$ We use again strong convexity, but this time we focus on the function $S\ni w\mapsto \|\pi_K w\|^2.$ Note that by optimality we have
\begin{equation*}
E(\hat w^*)\leq E(w_n^*)+\eta_n \|\pi_K w_n^*\|^2\leq E(\hat w^*)+\eta_n \|\pi_K \hat w^*\|^2,
\end{equation*}
so that
\begin{equation}
0\leq E(w_n^*)-E(\hat w^*)\leq \eta_n\left(\|\pi_K \hat w^*\|^2- \|\pi_K w_n^*\|^2 \right).
\end{equation}
This in particular implies that $\|\pi_K w_n^*\|^2\leq \|\pi_K \hat w^*\|^2$ and passing to the limit as $n\to +\infty$ we get 
$$\|\pi_K \bar w\|^2\leq \|\pi_K \hat w^*\|^2=\min\{\|\pi_K w^*\|^2, w^*\in S\}.$$
By strong convexity of $\|\pi_K \cdot\|^2$ on $S$ we can conclude that $\bar w=\hat w^*.$
\end{proof}
In view of the above proposition, when \eqref{transversality} does not hold, it is worth studying the following problem instead of Problem \ref{P3}.
\begin{problem}\label{P4}
For $\eta>0$ find $z_\eta^*\in \reals^M$ such that
$$F_\eta(z_\eta^*)=\min_{z\in \reals^M}F_\eta(z),$$
where
$$F_\eta(z):=F(z)+\eta\|\pi_K\Sqm{z}\|^2.$$
\end{problem}

\subsection{\emph{A posteriori} error bounds for problems \ref{P3} and \ref{P4}}\label{sub:error}
Well-conditioning for optimization problems is often understood in terms of \emph{a posteriori} error bounds, which tipically give an upper bound for the error by means of a residual term, e.g., norm of the gradient of the objective, possibly combined with terms depending on the constraints. For the \emph{unconstrained} minimization problems \ref{P3} and \ref{P4} upper bounds for the error can be derived through the \emph{{\L}ojasiewicz Inequality} (see \cite{Lo63}). {\L}ojasiewicz' Theorem asserts that such inequality holds true for any real analytic function at any point of its domain. Precisely, if $D\subseteq \reals^M$ is an open set and $f:D\rightarrow\reals$ is real analytic, then, for any $x\in D$, there exists an open neighborhood $U$ of $x$, a real number $\vartheta\in(0,1/2]$ (called the \emph{{\L}ojaciewicz exponent} of $f$ at $x$), and $C>0$ such that 
\begin{equation}
\label{eq:Loja}
\mid f(x)-f(y)\mid^{1-\vartheta}\leq C\mid\nabla f(y)\mid,\;\;\forall y\in U.
\end{equation} 
In order to study error bounds for Problem \ref{P3} we introduce an additional hypothesis that we may or may not assume.
\begin{setass}\label{hypP3}
There exists $z^*\in \argmin F$ such that 
\begin{align}
&[\partial^2_{i,j} E(\Sqm{z^*})]_{\{i,j:z_i^*\neq 0\neq z_j^*\}}\succ 0\label{SAp31}\\
&\partial_i E(\Sqm{z^*})\neq 0,\;\;\forall i: z_i^*=0. \label{SAp32}
\end{align}
\end{setass}
We are able to prove the following.
\begin{proposition}[Error bounds for Problem \ref{P3}]\label{prop:errorbound3}
There exist $R>0$, $C>0$, and $\vartheta\in (0,1/2]$, such that, for any $z^*\in \argmin F$, we have
\begin{equation}\label{eq:polycond}
\mid z-z^*\mid< \left(\frac{\mid\nabla F(z)\mid}C  \right)^{\frac 1{2(1-\vartheta)}},\;\;\forall z\in B(z^*,R).
\end{equation}
If the Set of Assumptions \ref{hypP3} holds, then we can take $\vartheta=1/2$, i.e., we have
\begin{equation}\label{eq:linearcond}
\mid z-z^*\mid< \frac{\mid\nabla F(z)\mid}C  ,\;\;\forall z\in B(z^*,R).
\end{equation} 
\end{proposition} 
\begin{proof}
Let us pick $z^*\in \argmin F$ and $z^0\in B(z^*,R).$ Let $t\mapsto z(t)$ be the unique real analytic solution of 
$$\begin{cases}
\dot z=-\nabla F(z)& t>0\\
z(0)=z^0
\end{cases},$$
and let us assume that $z^*=\lim_{t\to +\infty}z(t).$ Existence, uniqueness, real analyticity, and long-time behaviour of this solution will be proved in Theorem \ref{thm:Gf}. We stress that the proof of such a result is not depending on what we are proving here nor on the rest of the present subsection. Notice also that, if $z^*=\lim_{t\to +\infty}z(t)$ does not hold true, it would be sufficient to pick a smaller $R$. Let us suitably pick $L$ and $\vartheta\in(0,1/2].$  Thus we can write
\begin{align*}
&\mid z(t)-z^*\mid^2\leq \int_t^{+\infty}\mid\dot z(t)\mid^2ds\\
=& \int_{+\infty}^t\langle \nabla F(z(s));\dot z(s)\rangle ds= F(z(t))-F(z^*)\\
=& \mid F(z(t))-F(z^*)\mid\leq \left(\frac{\mid \nabla F(z)\mid}C  \right)^{\frac 1{1-\vartheta}}.
\end{align*}
If the Set of Assumptions \ref{hypP3} holds true, then the Hessian of $F$ at $z^*$ is not degenerate as it is easy to verify using \eqref{hessianF}, \eqref{SAp31}, and \eqref{SAp32}.

Due to e.g., \cite[Prop. 2.2]{MePi10}, the {\L}ojaciewicz Inequality holds for $F$ in a suitable neighbourhood of $z^*$ with $\vartheta=1/2.$
\end{proof}
When \eqref{transversality} does not hold and we consider Problem \ref{P4} instead of the ill-posed Problem \ref{P2} we have a similar result.
\begin{proposition}[Error bounds for Problem \ref{P4}]\label{prop:errorbound4}
For any $\eta>0$ and $z_\eta^*\in \argmin F_\eta$ there exist $R_\eta,C_\eta>0$ and $\vartheta\in (0,1/2]$ such that
\begin{equation}\label{epsilonbound1}
\mid z-z_\eta^*\mid< \left(\frac{\mid\nabla F_\eta(z)\mid}{C_\eta}\right)^{\frac{1}{2(1-\vartheta)}}  ,\;\;\forall z\in B(z_\eta^*,R_\eta).
\end{equation}
In particular, if $\eta_n\to 0$ and  $\hat z^*:=\lim_n z_{\eta_n}^*$, we have 
\begin{equation}\label{epsilonbound2}
\mid \hat z^*-z_{\eta_n}^*\mid\leq \left(\frac{2\eta_n\|z^*\|}{C_{\eta_n}}\min_{w\in S}\|\pi_K w\|\right)^{\frac{1}{2(1-\vartheta)}}, \text{ for }n\text{ large enough}.
\end{equation}
Moreover, if
\begin{equation}
\partial_i E(\Sqm{z^*_\eta})\neq -2\eta(\pi_K\Sqm{z^*_\eta})_i\; \forall i:\,(z^*_\eta)_i=0,\label{H4}
\end{equation}
then $\vartheta=1/2$, i.e.,
\begin{equation}\label{epsilonbound3}
\mid z-z_\eta^*\mid< \frac{\mid\nabla F_\eta(z)\mid}{C_\eta}  ,\;\;\forall z\in B(z_\eta^*,R_\eta),
\end{equation} 
and 
\begin{equation}\label{epsilonbound4}
\mid \hat z^*-z_{\eta_n}^*\mid\leq \frac{2\eta_n\|\hat z^*\|}{C_{\eta_n}}\min_{w\in S}\|\pi_K w\|^2, \text{ for }n\text{ large enough}.
\end{equation}
\end{proposition}
\begin{proof}
The proof of \eqref{epsilonbound1} is identical to the one of \eqref{eq:polycond}. To obtain \eqref{epsilonbound2} from \eqref{epsilonbound1}, it is sufficient to note that
\begin{equation}
\mid\nabla F_\eta(z^*)\mid=\mid \nabla F(z^*)+\eta \nabla\left[ \|\pi_K\Sqm{z^*}\|^2\right]\mid=\eta \mid\nabla\left[ \|\pi_K\Sqm{z^*}\|^2\right]\mid,
\end{equation}
and that
\begin{equation}
\mid \nabla\left[ \|\pi_K\Sqm{z^*}\|^2\right]\mid=2\mid\langle \pi_K z^*;\pi_K\Sqm{z^*}\rangle \mid\leq 2\|z^*\|\|\pi_K\Sqm{z^*}\|. 
\end{equation} 
We now show that, if \eqref{H4} holds, then $\Hess F_\eta(z^*_\eta)$ is positive definite. This would imply that we can take $\vartheta=1/2$, as pointed out in the proof of Proposition \ref{prop:errorbound3}. Let $P_K$ the matrix representing the orthogonal projection onto $K$. We can compute
\begin{align*}
\Hess F_\eta(z^*_\eta)=&4\diag z^*_\eta (\Hess E(\Sqm{z^*_\eta})+2\eta P_K^tP_K)\diag z^*_\eta+\\
&\;\;\;\;\;\;\;\;\;\;2 \diag(\nabla E(\Sqm{z^*_\eta})+2\eta P_K^tP_K \Sqm{z^*_\eta})\\
=&4\diag z^*_\eta (\Hess E(\Sqm{z^*_\eta})+2\eta P_K^tP_K)\diag z^*_\eta+\\
&\;\;\;\;\;\;\;\;\;\;2 \diag(\nabla E(\Sqm{z^*_\eta})+2\eta \pi_K \Sqm{z^*_\eta})\\
=:&H_1+H_2.
\end{align*}
It is clear that $H_1$ is a positive deifite matrix on the linear space $V(z^*_\eta)$ generated by $\boldsymbol e_{i_1},\dots,\boldsymbol e_{i_p}$, where  $(z_\eta^*)_{i_1},(z_\eta^*)_{i_2},\dots, (z_\eta^*)_{i_p}\neq 0.$ 

Recall that, since $w^*_\eta:=\Sqm{z^*_\eta}$ is a minimizer for the objective functional $E_\eta$ on $\realspos^M$, it satisfies the Karush Kuhn Tucker conditions
\begin{equation}\label{KKTepsilon}
\begin{cases}
\partial_i E(\Sqm{z^*_\eta})+2\eta(\pi_K\Sqm{z^*_\eta})_i\geq 0& \text{ if }(z^*_\eta)_i=0\\
\partial_i E(\Sqm{z^*_\eta})+2\eta(\pi_K\Sqm{z^*_\eta})_i= 0& \text{ otherwise}
\end{cases}.
\end{equation}
Equations \eqref{H4} and \eqref{KKTepsilon} imply that $H_2$ is positive semi-definite and the associated quadratic form vanishes only on $V(z^*_\eta),$ where the quadratic form canonically associated to $H_1$ is positive. Thus $H_1+H_2$ is positive definite.
\end{proof} 

\section{Optimization algorithm for Problem \ref{P3} from log-determinant gradient flow}
\subsection{Derivation of the algorithm}
The algorithm we are going to define can be derived as an infinite horizon integrator for the gradient flow $\dot z=-\nabla F(z)$. For this reason we show first that the gradient flow of $F$ enjoys existence, unicity, and regularity properties.
\begin{theorem}[The gradient flow of $F$]\label{thm:Gf}
Let $z^0\in \reals^M$ be in the domain $\Dom(F)$of $F$ and such $\nabla F(z^0)\neq 0.$ Then there exists a globally real analytic solution (both in parameter and in time) $z(\cdot;z^0):[0,+\infty[\rightarrow \reals^M$ of the equation
\begin{equation}
\label{GF}
\begin{cases}
\dot z(t)=-\nabla F(z(t))& t>0\\
z(0)=z^0
\end{cases}\;.
\end{equation}
There exists $z^*(z^0)\in \argmin{}_{\reals^M} F$ such that $\lim_{t\to +\infty}  z(t;z^0)=z^*(z^0)$. Thus $\Sqm{z(t;z^0)}$ converges to a D-optimal design.
\end{theorem}
\begin{proof}
We can devide the proof in few steps.
\begin{itemize}
\item []\emph{Step 1:} local existence, uniqueness, and $\mathscr C^1$ regularity.
\end{itemize}
A local solution to \eqref{GF} can be constructed by Peano Lindeloff Theorem. Let $t_0=0$, $c'> F(z^0)$,  $\epsilon>0$, and $c>\sup\{ F(z): z\in \cup_{ F(\eta)<c'}B(\eta,\epsilon)\}$. Let us set
\begin{align*}
L_c&:=\sup\{\rho(\Hess  F(z)),\, F(z)<c\}\\
M_c&:=\sup\{\|\nabla  F(z)\|,\, F(z)<c\}\\
\Delta t&:=\max\{1/L_c,\epsilon/M_c\}.
\end{align*} 
Here we used the notation $\rho(A)$ for the spectral radius of the matrix $A$. By the Picard Lindelof Theorem there exists a unique $\mathcal C^1$ curve $z:(t_0-\Delta t,t_0+\Delta t)\rightarrow B(z^0,\epsilon)$ such that \eqref{GF} holds true.
\begin{itemize}
\item [] \emph{Step 2:} global existence, uniqueness and $\mathscr C^\omega$ regularity.
\end{itemize}
Let us set 
$$t_{i+1}:=t_i+\Delta t/2,\;\;\;z^{(i+1)}:=z(t_{i+1}).$$
Note that
\begin{align*}
F(z^{(1)})=& F(z^{(0)})+\int_{t_0}^{t_1} \langle\nabla F(z(s));z'(s)\rangle ds\\
=& F(z^{(0)})-\int_{t_0}^{t_1}\|\nabla  F(z(s))\|^2 ds
\leq   F(z^{(0)})\leq c.
\end{align*} 
Therefore we can repeat the argument above to provide the existence of a curve $\hat z:(t_1-\Delta t,t_1+\Delta t)=(t_0-\Delta t/2,t_0+3/2 \Delta t)\rightarrow B(z^{(1)},\epsilon)$ such that \eqref{GF} holds true with $z^0$ replaced by $z^1.$ However $z(t_1)=\hat z(t_1)$ and the above proven uniqueness property shows that the two curves coincide in $(t_1-\Delta t/2,t_1+\Delta t/2).$ Thus we can redefine $z:(t_0-\Delta t, t_1+\delta t)\rightarrow \reals^M$ by gluing the two curves. The curve $z$ inherits the $\mathscr C^1$ regularity of $\hat z$ and $\bar z$ due to the continuity of $F.$  Repeating the above calculations we get $F(z^{(2)})<c$ and in general
$$F(z^{(i+1)})\leq F(z^{(i)})<c,\;\;\forall i\in \naturals.$$
Iterating the procedure we construct a unique $\mathscr C^1$ solution $z(\cdot;\sigma^0):(0,+\infty)\rightarrow \Dom  F$ of \eqref{GF}. This solution is actually real analytic both with respect to the time variable $t$ and the initial data $z^0\in \Dom F$ in view of Cauchy Kowalewskaya Theorem.
\begin{itemize}
\item [] \emph{Step 3:} existence of long time asymptotics $z^*$ and characterization as critical point of $F$.
\end{itemize}
By \eqref{GF}, and being $F$ bounded from below, we can write
\begin{equation}\label{EqEnergyF}
-\infty\geq F(z(t;z^0))- F(z^0)=-\int_0^{+\infty}\|\nabla F(z(s;\sigma^0))\|^2ds=-\int_0^{+\infty}\|z'(s;z^0))\|^2ds.
\end{equation}
Thus in particular $\|z'(\cdot;z^0))\|_{L^2(0+\infty)}<+\infty.$ It is not difficult to conclude that there exists $z^*\in \Dom  F$ such that 
$$\lim_{t\to +\infty}z(t;z^0)=z^*.$$
Using again \eqref{EqEnergyF} and the continuity of $\nabla  F$ we have
$$\nabla  F(z^*)= \lim_{t\to +\infty}\nabla F(z'(t;z^0))=0.$$
Therefore $z^*$ is a critical point of $ F.$ We are left to prove that $w^*:=\Sqm{z^*}$ is a minimizer of $E$ (and $z^*$ is a minimizer of $ F$).  

\begin{itemize}
\item [] \emph{Step 4:} the critical point $z^*$ is a global minimizer.
\end{itemize}
Since $z^*$ is critical, we have $0=\partial_i F(z^*)=2z_i\partial_i E(\Sqm{z^*})$ for any $i$. Thus 
\begin{equation}\label{kktfirstpart}
\partial_i E(w^*)=0\text{ for all }i\text{ such that }w^*_i\neq 0.
\end{equation}
 Note that this is part of the Karush Kuhn Tucker conditions for the minimization of $E$, see \eqref{kktP2}. 

Now we claim that, for any $i$ such that $w_i^*=0$, we can find a monotone increasing sequence $\{t_j\}_{j\in \naturals}$, $t_j\to +\infty$, such that, for any $j\in \naturals$, we have
\begin{equation}\label{claim}
\begin{cases}
z'(t_j;z^0)\geq 0& \\
z(t_j;z^0)\leq 0
\end{cases}
\;\;\;\;\text{ or }\;\;\;
\begin{cases}
z'(t_j;z^0)\leq 0& \\
z(t_j;z^0)\geq 0
\end{cases}.
\end{equation}
In such a case, using $\partial_i F(z(t_j;z^0))=-z'(t_j;z^0)$ and $\sign \partial_i E(\Sqm{z(t_j;z^0)})=\sign \partial_i F(z(t_j;z^0)) \sign z_i(t_j;z^0)$, it is easy to see that $\partial_i E(\Sqm{z(t_j;z^0)})\geq 0$ for all $j$ and all $i$ as above. Therefore 
\begin{equation}\label{kktsecondpart}
\partial_i E(w^*)=\lim_j\partial_i E(\Sqm{z(t_j;z^0)})\geq 0,\;\;\forall i: z_i^*=0.
\end{equation}
Note that \eqref{kktfirstpart} and \eqref{kktsecondpart} are precisely the Karush Kuhn Tucker conditions for the minimization of $E$ on $\realspos^M.$ Thus $w^*$ is a minimizer of $E$ and, due to Theorem \ref{thmequivalenve}, $z^*$ is a minimizer of $F.$

The proof is concluded if we show \eqref{claim}. This can be done by contradiction. If \eqref{claim} is not satisfied, there exists $T>0$ such that $\sign z_i(t;z^0)=\sign z_i'(t;z^0)$ for all $t\geq T.$ Recall that
$$z^*_i= z_i(t;z^0)+\int_t^{+\infty} z_i'(s;z^0) ds,\;\;\forall t\geq T.$$ 
Therefore, if $z_i(t;z^0)$ is positive for some $t\geq T$, then $z^*_i$ is positive, conversely if $z_i(t;z^0)$ is negative for some $t\geq T$, then $z^*_i$ is negative. However we are assuming $z^*_i=0$, thus \eqref{claim} needs to hold.
\end{proof}
We now derive our algorithm by numerical integration of this flow. Recall that we are not aiming at contructing a good approximation of the trajectories for a finite time interval, possibly loosing accuracy as the time variable grows large. We are rather interested in constructing discrete trajectories that inherit the variational properties of the flow of $F$ and, in particular, have the same attraction basins. 

To accomplish this purpose we combine the \emph{Backward Euler Scheme} with a bound on the time step with the (zero finding) \emph{Newton's Method} with prescribed initial guess and a particular stopping criterion. These choices are made to ensure that certain qualitative properties of the scheme hold true. Then we use such properties for proving that the derived algorithm is indeed convergent.

Among these properties the most relevant are the following. 
\begin{itemize}
\item The  upper bound for the time step $\tau$ depends only on the level of $F$ at the starting point $z^0.$
\item The Backward Euler Scheme for computing $z^{k+1}$ is indeed a variational scheme, i.e., can be written as the minimum problmem for the locally convex objective $g(\cdot;z^{k},\tau)$, where
\begin{equation}\label{variationalscheme}
g(z;z^{k},\tau):=F(z)+\frac{\mid z-z^{k}\mid^2}{2\tau}.
\end{equation}
\item The initial guess of the Newton's Method is in the attraction basin of the minimizer of the aforementioned optimization problem.
\item The (first condition in the) stopping criterion forces the discrete trajectories to preserve a variational property from which we can derive a stability estimate.
\item The (second condition in the) stopping criterion prevents the saddle points of $F$ to became attractors of the discrete-time flow.
\end{itemize}
The choices we made lead to Algorithm \ref{algo1} below.

\begin{algorithm}
\caption{Compute D-optimal design}\label{algo1}
\begin{algorithmic}[0]
\State{Input $z^0\in \reals^M$: $\nabla F(\Sqm{z^0})\neq 0$, $\tau>0$, $n_{step}\in \naturals$, $toll>0$, $\epsilon>0$ }
\State{Set $k:=0$}
\State{Compute $res=\mid\nabla F(z^0)\mid$}
\If{$res=0$}
\State{\textbf{Exit} with error.}
\EndIf
\State{$z^{old}=z^0$}
\While{$k<n_{step} \vee res>toll$}
\State{Set $k=k+1$, $z^{new}:=z^{old}$}
\State{Compute $res_{Newton}:=\nabla g(z^{new};z^{old},\tau)$  }
\While{$\mid(res_{Newton})_i\mid>\epsilon \mid(z^{new}-z^{old})_i\mid$ for some $i \wedge \sign z^{new}\neq \sign z^{old}$}
\State{Compute $z^{new}=z^{new}-[\Hess g(z^{new};z^{old},\tau)]^{-1}\nabla g(z^{new};z^{old},\tau) $}
\State{Compute $res_{Newton}:=\nabla g(z^{new};z^{old},\tau)$}
\EndWhile
\State{Compute $res=\mid\nabla F(z^{new})\mid$}
\EndWhile\\
\Return $\Sqm{z^{new}}$
\end{algorithmic}
\end{algorithm}

\vfill

\subsection{Convergence analysis for Algorithm \ref{algo1}}\label{subs:convergenceanalysis}
As a first step we prove that the algorithm is consistent. Precisely, if at each $k$-th stage we allow the Newton's method to run an infinite number of times, then it converges to a point $z^{k+1}$ that minimizes $g(\cdot;z^{k},\tau)$.
\begin{theorem}[Concistency of Algorithm \ref{algo1}]\label{thm:concistency}
Let $c>\min_{\reals^M} F$. There exists $\tau^*>0$ (depending only on $c$) such that the following holds true for any $\tau<\tau^*$.
\begin{enumerate}[i)]
\item For any $z^0\in \reals^M$ be such that $\nabla F(z^0)\neq 0$ the following sequences are well defined
\begin{align}
z^{k+1,0}&:=z^{k}\\
z^{k+1,r+1}&:=z^{k+1,r}-[\Hess g(z^{k+1,r};z^{k},\tau)]^{-1}\nabla g(z^{k+1,r};z^{k},\tau)\\
z^{k+1}&:=\lim_r z^{k+1,r},
\end{align} 
\item For any $k\in \naturals$ we have
\begin{equation}
\nabla F(z^{k+1})=-\frac{z^{k+1}-z^k}{\tau}.
\end{equation}
\item There exists $z^*\in \argmin F$ such that
\begin{equation}
\lim_k z^k=z^*,
\end{equation}
in particular $\Sqm{z^k}$ converges to a D-optimal design $w^*:=\Sqm{z^*}$.
\end{enumerate}
\end{theorem}
\begin{proof}
The proof of the statements \emph{i)} and \emph{ii)}  can be obtained following the lines of the classical proof of convergence of the Newton's Method. For this reason we only sketch this part of the proof, highlighting the overall technique and the main estimates that are needed, but leaving few details to the reader.

Let $\Omega_k$ be the connected component of the set $\{F_<F(z^k)\}$ containing $z^k$. Let us set 
\begin{align}
& d_k:=diam(\Omega_k),\\
&U_k:=\cup_{z\in \Omega_k}B(z,d),\\
& \gamma:=\max_{y\in U_k}\max_{z\in \Omega_k}\mid\nabla g(y;z,\tau)\mid,\\
&\lambda:=\min_{y\in U_k} \lambda_{min}(\Hess F(y)),\\
&R:=\left(\max_{z\in U_k}\rho(\Hess \partial_1F(z)),\dots,\max_{z\in U_k}\rho(\Hess \partial_{M}F(z))\right).
\end{align}
Notice that the function $g(\cdot;y,\tau)$ has Hessian matrix independent by $y$:
$$\Hess g(z;y,\tau)=\Hess F(z)+\frac 1 \tau \mathbb I.$$
Therefore the function $g(\cdot;y,\tau)$ is strongly convex on $U_k$ for any $y\in\Omega_k$, provided
\begin{equation}\label{taubound1}
\tau< \frac 1{\lambda^-}=:\tau_1.
\end{equation}
Assuming \eqref{taubound1} we denote by $\hat z^{k+1}$ the unique minimizer of $g(\cdot;z^k,\tau)$ in $U_k$. Then it follows by the definition of $g(\cdot;z^k,\tau)$ and its relation with $F$ that
\begin{equation}\label{functionvsstep2}
F(z^{k})-F(\hat z^{k+1})\geq \mid\hat z^{k+1}-z^k\mid^2/(2\tau).
\end{equation}
Using the strong convexity of $g(\cdot;z^k,\tau)$ of parameter $(\lambda+1/\tau)$ we can show that $\hat z^{k+1}$ lies in $\Omega_k.$

Let us introduce the notation
\begin{equation*}
e^{k+1,r}:=z^{k+1,r}-\hat z^{k+1},\;\;\;s^{k+1,r}:=z^{k+1,r}-z^{k+1,r-1}.
\end{equation*}
Writing the second order Taylor expansion of $0=\nabla g(\hat z^{k+1};z^k,\tau)$ centered at $z^{k+1,r}$ we obtain the standard estimate
\begin{equation}\label{forquadratic}
\mid e^{k+1,r+1}\mid\leq \frac{R\tau}{2(\lambda\tau+1)}\mid e^{k+1,r}\mid^2.
\end{equation} 
On the other hand, writing the first order Taylor expansion of $0=\nabla g(\hat z^{k+1};z^k,\tau)$ centered at $z^{k+1,0}=z^k$ we get
\begin{equation}\label{forlinear}
\mid e^{k+1,r}\mid\leq \frac{\tau}{\lambda\tau+1}\mid \nabla g (z^{k};z^k,\tau)\mid\leq \frac{\tau\gamma}{\lambda\tau+1}.
\end{equation}
Let us assume
\begin{equation}
\label{taubound2}
\tau<\frac{2\lambda+\sqrt{2 R\gamma}}{(R\gamma-2\lambda^2)}=:\tau_2,
\end{equation}
then $\frac{R\tau}{2(\lambda\tau+1)}\mid e^{k+1,r}\mid<1$, thus, due to \eqref{forquadratic}, we have
\footnotesize
\begin{equation}
\label{linearnewton}
\mid e^{k+1,r+1}\mid\leq \frac{R\tau}{2(\lambda\tau+1)}\mid e^{k+1,r}\mid^2<\mid e^{k+1,r}\mid<\ldots <\mid e^{k+1,0}\mid=\mid\hat z^{k+1}-z^k\mid<d,
\end{equation}
\normalsize
In particular $z^{k+1,r+1}$ is in $\Omega_k$ and we can repeat the argument above. Then using iteratively \eqref{forquadratic} we obtain the quadratic convergence of $z^{k+1,r}$ to $\hat z^{k+1},$
provided $\tau< \tau^*:= \min\{\tau_1,\tau_2\}.$

The proof of \emph{iii)} can be devided in two steps: first we show that $z^{k}$ converges to a critical point $z^*$ for $F$, second we show that $z^*$ is indeed a global minimizer for $F$ and hence $w^*:=\Sqm{z^*}$ is an optimal design.

Let us notice that we can exclude the case $\nabla F(z^k)=0$. Indeed, by the definition of the sequence $\{z^k\}$ we have $\nabla g(z^{k};z^{k-1},\tau)=0$, so $z^{k-1}=z^{k}$. By finite induction we get $z^0=z^k$. Hence $\nabla F(z^k)=0$, which contradicts our hypothesis. 
Using \ref{functionvsstep2} and \emph{ii)} we can write
\begin{equation}
F(z^j)-F(z^{j+1})\geq \frac{\mid z^j-z^{j+1}\mid^2}{2\tau}=\frac{\tau}{2}\mid\nabla F(z^{j+1})\mid^2.
\end{equation}
Summing up over $j=0,1,\dots, k-1$ we get 
\begin{equation}\label{provecauchy}
F(z^0)-F(z^{k})\geq\sum_{j=0}^{k-1} \frac{\mid z^j-z^{j+1}\mid^2}{2\tau}=\frac{\tau}{2}\sum_{j=0}^{k-1}\mid\nabla F(z^{j+1})\mid^2.
\end{equation}
Being $F$ bounded from below, \eqref{provecauchy} in particular shows that $\{z^k\}$ is a Cauchy sequence and its limit, say $z^*$, is critical for $F$.

We are left to prove that $w^*:=\Sqm{z^*}$ is indeed a minimizer for $E$. There are two cases to be considered
\begin{enumerate}[a)]
\item $z^*_i\neq 0 \forall i=1,1,\dots, M$
\item There exixts $\emptyset \neq I\subset \{1,2,\dots, M\}$ such that $z^*_i=0$ if and only if $i\in I.$
\end{enumerate}
Case a) is easier to be discussed. Indeed, by since $\partial_i E(w^*)=(2 z^*_i)^{-1}\partial_i F(z^*)$, we obtain $\nabla E(w^*)=0$. Since $E$ is convex (see Proposition \ref{}), we can conclude that $w^*$ is a global minimizer of $E$ (and $z^*$ is a global minimizer of $F$). 

Case b) is slightly more complicated. First we note that, for any $i=1,2,\dots,M$ the sequence $\{z^{k+1}_i\}_k$ must have constant sign. The proof of this statement easily follows by the strong convexity of $g(\cdot;z^k,\tau)$ and its the symmetry.

Now pick any $i\in I$. We claim that we can pick a subsequence $j\mapsto k_j$ such that 
\begin{equation}\label{claimsign}
\sign z^{k_j}_i=\sign \partial_i F(z^{k_j}).
\end{equation}
Also this claim can be proven by contradiction. For, let us assume that we have $\sign z^{k}_i=-\sign \partial_i F(z^{k})$ for any $k$. Then, since 
\begin{equation}\label{toshowsigns}
z^{k_{j+1}}_i=z^{k_j}-\tau \partial_i F(z^{k_{j+1}})
\end{equation}
and $z^k_i$ has constant sign, it follows that the sequence $z^{k_j}$ is either positive and increasing or negative and decreasing, depending on the sign of $z^{k_J}.$ In both cases we cannot have $z_i^{k_j}\to 0$ and this is a contradiction since $i\in I$. Thus \eqref{claimsign} holds true.

Notice that by \eqref{claimsign} and  $\partial_i E(\Sqm{z^k})=(2 z^k_i)^{-1}\partial_i F(z^k)$ it follows that $\partial_i E(w^*)=\lim_k \partial_i E(\Sqm{z^k})\geq 0.$ Finally recall that $E$ is convex and we already prove that
\begin{equation}
\begin{cases}
w^*_i\geq 0& \forall i\\
\partial_i E(w^*)=0& \text{ if }w_i^*=0\\
\partial_i E(w^*)\geq 0& \text{ if }w_i^*\neq 0
\end{cases}.
\end{equation}
These equations are precisely the Karush Kuhn Tucker sufficient conditions for $w^*$ minimizing $E$ over $\realspos^M.$ Note that this shows that $z^*$ is a minimizer of $E$ as well.
\end{proof}
\begin{remark}
It is worth stressing that both the continuous time gradient flow \eqref{GF} and the discrete trajectories constructed by means of a variational scheme as \eqref{variationalscheme} are not in general converging to a (even local) minimizer of the objective functional. Indeed in the case of a non-convex objective (as in our case) the attractor of the flow may contain stationary points. Here the convergence both for the continuous time (see Theorem \ref{thm:Gf}) and discrete trajectories (see Theorem \ref{thm:concistency}) follows from the specific structure of $F$, which is the composition of a convex functional and the coordinate square map.
\end{remark}
In order to continue our study of Algorithm \ref{algo1}, it is convenient to introduce some notations. Let us denote by 
$$\Psi(\cdot;\tau):\reals^M\rightarrow \reals^M$$
the map that, for any $z^k\in \reals^M$, returns the exact solution of $\nabla g(\cdot;z^k,\tau)=0$ provided by the Newton's Method with $z^k$ as initial guess. As a biproduct of Theorem \ref{thm:concistency}  this map is well defined, provided $\tau>0$ is sufficiently small. We also define the map $\Psi_\epsilon(\cdot;\tau):\reals^M\rightarrow \reals^M$ as
$$\Psi_\epsilon(z_k;\tau):=z^{k+1,\hat r},$$
where $\hat r$ is defined by the stopping criterion for the Newton's Method appearing in Algorithm \ref{algo1}, i.e.,
\footnotesize
\begin{equation}
\label{rdef}
\hat r:=\min\left\{r: \mid\partial_i g(z^{k+1,r})\mid\leq \epsilon \mid(z^{k+1,r})_i-(z^k)_i\mid,\;\sign (z^{k+1,r})_i=\sign(z^k)_i\,\forall i\right\}.
\end{equation}
\normalsize
We remark that, given an intial point $z^0$ and suitable $\epsilon,\tau>0$, the Algorithm \ref{algo1} computes the finite sequence
$$\left(z^0, \Psi_\epsilon(z^0;\tau), \Psi^2_\epsilon(z^0;\tau),\Psi^3_\epsilon(z^0;\tau),\dots\right)$$
of length at most $n_{step}$.
\begin{proposition}[Stability estimate for Algorithm \ref{algo1}]
\label{prop:stability}
Let $\min F<c<+\infty$ and let $\tau<\tau^*(c)$ be as above. There exists  $\epsilon^*>0$, depending only on $c$ and $\tau$, such that 
\begin{equation}
F(z)-F(\Psi_\epsilon(z;\tau))\geq \frac {F(z)-F(\Psi(z;\tau))} 2,\;\;\forall z:F(z)<c, 
\end{equation}
for any $0<\epsilon\leq\epsilon^*$.
\end{proposition}
\begin{proof}
It is convenient to introduce the notation
\begin{equation*}
  \Delta:= F(z)-F(\Psi(z;\tau)),\;\;\;\Delta_\epsilon
  :=F(z)-F(\Psi_\epsilon(z;\tau)).
\end{equation*}
Let us pick $0<\epsilon<\epsilon_1:=\lambda+1/\tau$. Using the standard error bound for the Newton's Method we can write
\begin{equation}\label{ebound}
  \mid e_\epsilon\mid:=\mid\Psi_\epsilon(z;\tau)-\Psi(z;\tau)\mid
  \leq \frac{\tau\epsilon}{\lambda\tau+1}\mid s_\epsilon\mid.
\end{equation}
Also by the triangular inequality we have
\begin{equation*}
  \mid s_\epsilon\mid
  \leq\frac{\lambda\tau+1}{(\lambda-\epsilon)\tau+1}\mid s\mid
  :=C_{\epsilon}\mid s\mid\, ,
\end{equation*}
which yields:
\begin{equation*}
    \mid s_\epsilon\mid\leq
    \frac{C_{\epsilon}\epsilon\tau}{\lambda\tau+1}\mid s\mid\, .
\end{equation*}
Using the second order Taylor expansion of $g(z;)$ centered at
$\Psi(z;\tau)$ we can obtain
\begin{equation}\label{sbound}
  \Delta\geq \frac{\lambda\tau+2}{2\tau}\mid s\mid^2.
\end{equation}
On the other hand, we can write
\begin{align*}
  \Delta_\epsilon
  &=\Delta-(F(\Psi(z;\tau))-F(\Psi_\epsilon(z;\tau)))
    =\Delta+\langle e_{\epsilon},\frac{s}{\tau}\rangle -
    \frac{1}{2} e_{\epsilon}^T\Hess F(\xi) e_{\epsilon} \\
  &\geq \Delta - \frac{\mid s\mid\mid e_{\epsilon}\mid}{\tau}
    -\frac{\Lambda}{2}\mid e_{\epsilon}\mid^2 
    \geq \Delta - \mid s\mid^2
    \left(
    C_{\epsilon}\frac{\epsilon}{\lambda\tau+1}+
    \frac{\Lambda
    C_{\epsilon}^2}{2}\frac{(\epsilon\tau)^2}{(\lambda\tau+1)^2}
    \right) \\
  &\geq \Delta
    \left[
    1-\frac{2\tau
    C_{\epsilon}\epsilon}{(\lambda\tau+2)(\lambda\tau+1)}
    \left(
    1+\frac{\Lambda\tau^2C_{\epsilon}\epsilon}{\lambda\tau+1}
    \right)
    \right]
\end{align*}
In order to conclude the proof, we are left to verify that for small
$\epsilon>0$ we have
\begin{equation*}
    1-\frac{2\tau
    C_{\epsilon}\epsilon}{(\lambda\tau+2)(\lambda\tau+1)}
    \left(
    1+\frac{\Lambda\tau^2C_{\epsilon}\epsilon}{\lambda\tau+1}
    \right)\geq\frac{1}{2}.
\end{equation*}
Thus
\begin{equation}\label{conditiononC}
  C_{\epsilon}\epsilon\leq
  \left(\sqrt{1+\Lambda\tau(\lambda\tau+2)}-1\right)
  \frac{\lambda\tau+1}{2\Lambda\tau^2}\,.
\end{equation}
and note that $C_{\epsilon}\epsilon=\epsilon+o(\epsilon)$ as
$\epsilon\to 0^+$. Thus we can pick $\epsilon_2$ such that \eqref{conditiononC} holds for any $\epsilon<\epsilon_2.$ Finally we set $\epsilon^*:=\min\{\epsilon_1,\epsilon_2\}.$
\end{proof}
We can now prove the convergence of Algorithm \ref{algo1} by combining Proposition \ref{prop:stability}, Theorem \ref{thm:concistency}, and the technique used in the end of the proof of Theorem \ref{thm:Gf}.

\begin{theorem}[Convergence of Algorithm \ref{algo1}]
\label{thm:convergence}
Let $\min F<c<+\infty$ and $z^0$ such that $F(z^0)<c,$ $\nabla F(z^0)\neq 0$. Let $\tau>0$ and $\epsilon$ satisfy the hypothesis of Proposition \ref{prop:stability}. Then the sequence $\{[\Psi_\epsilon]^{(k)}(z^0;\tau)\}_{k\in \naturals}$ admits a limit $z^*\in \argmin F$. Thus
\begin{equation}
w^*:=\Sqm{(z^*)}\in \argmin E,
\end{equation}
i.e., Algorithm \ref{algo1} converges to an optimal design.
\end{theorem}

\begin{proof}
If $\nabla F([\Psi_\epsilon]^{(k)}(z^0;\tau))=0$ for certain $k$, then, using \eqref{rdef}, we can repeat the argument of the proof of Theorem \ref{thm:concistency}, to show that $[\Psi_\epsilon]^{(k)}(z^0;\tau)\equiv z^0$, which is not possible since we are assuming $\nabla F(z^0)\neq 0.$ Therefore we can assume without loss of generality that $\nabla F([\Psi_\epsilon]^{(k)}(z^0;\tau))\neq 0$ $\forall k\in \naturals.$ 

We use Proposition \ref{prop:stability} and the optimality of the exact step, i.e.,  $[\Psi]^{(k)}(z^0;\tau)-[\Psi]^{(k-1)}(z^0;\tau)=-\tau \nabla F([\Psi]^{(k)}(z^0;\tau))$, to get
\begin{align*}
&F([\Psi_\epsilon]^{(k)}(z^0;\tau))-F([\Psi_\epsilon]^{(k+1)}(z^0;\tau))\\
\geq &\frac 1 2 F([\Psi_\epsilon]^{(k)}(z^0;\tau))-F(\Psi([\Psi_\epsilon]^{(k)}(z^0;\tau));\tau))\\
\geq&\frac{1}{2\tau}\mid\Psi([\Psi_\epsilon]^{(k)}(z^0;\tau));\tau)-[\Psi_\epsilon]^{(k)}(z^0;\tau)  \mid^2.
\end{align*}
Notice that, using the notation introduced in the proof of Proposition \ref{prop:stability}, this last inequality can be written in the compact form 
\begin{equation}\label{estimateDeltas}
\Delta_\epsilon\geq \frac 1 2\Delta\geq \frac{\mid {s}\mid^2}{2\tau} .
\end{equation}
On the other hand, using \eqref{ebound}, we get 
$$\mid  s_\epsilon\mid=\mid s+e_\epsilon\mid\leq \frac{\lambda \tau+1}{(\lambda-\epsilon)\tau+1}\mid  s\mid.$$
Therefore we have
\begin{align*}
&F([\Psi_\epsilon]^{(k)}(z^0;\tau))-F([\Psi_\epsilon]^{(k+1)}(z^0;\tau))\geq \frac{[(\lambda-\epsilon)\tau+1]^2}{2\tau(\lambda\tau+1)}\mid  s_\epsilon\mid^2\\
=&\frac{[(\lambda-\epsilon)\tau+1]^2}{2\tau(\lambda\tau+1)^2}\mid[\Psi_\epsilon]^{(k+1)}(z^0;\tau)-[\Psi_\epsilon]^{(k)}(z^0;\tau)\mid^2.
\end{align*}
It is clear that, for any $k\in \naturals$ we have $F(z^0)-F([\Psi_\epsilon]^{(k)}(z^0;\tau))\leq c-\min F<+\infty$. Thus we can write
\begin{align*}
&+\infty>c-\min F\geq \lim_k\sum_{j=0}^{k-1}\left(F([\Psi_\epsilon]^{(j)}(z^0;\tau))-F([\Psi_\epsilon]^{(j+1)}(z^0;\tau))  \right)\\
\geq& \frac{[(\lambda-\epsilon)\tau+1]^2}{2\tau(\lambda\tau+1)^2}\sum_{j=0}^{+\infty}\mid[\Psi_\epsilon]^{(j+1)}(z^0;\tau)-[\Psi_\epsilon]^{(j)}(z^0;\tau)\mid^2
\end{align*}
Thus in particular $F([\Psi_\epsilon]^{(k)}(z^0;\tau))-F([\Psi_\epsilon]^{(k+1)}(z^0;\tau))\to 0$ as $k\to+\infty$ and $[\Psi_\epsilon]^{(k)}(z^0;\tau)$ is a Cauchy sequence: let us denote by $z^*$ its limit.

By an analog reasoning, starting from \eqref{estimateDeltas} we can show that 
$$\mid\Psi(\Psi_\epsilon^k(z^0;\tau);\tau)-\Psi_\epsilon^k(z^0;\tau)\mid\to 0,$$
thus $\Psi(\Psi_\epsilon^k(z^0;\tau);\tau)\to z^*.$ It follows by the definition of the map $\Psi$ that we have 
$$-\tau \nabla F(\Psi(\Psi_\epsilon^k(z^0;\tau);\tau))=\Psi(\Psi_\epsilon^k(z^0;\tau);\tau)-\Psi_\epsilon^k(z^0;\tau).$$
Therefore we have $\mid\nabla F(z^*)\mid=\lim_k\mid\nabla F(\Psi(\Psi_\epsilon^k(z^0;\tau);\tau))\mid=0,$ i.e., $z^*$ is critical for $F.$ 

We are left to show that $z^*$ is a minimizer of $F$ and $w^*$ is a global minimizer for $E.$  We reason as in the final step of the proof of Theorem \ref{thm:concistency}. The only needed modification is that, instead of equation \eqref{toshowsigns}, we need to use 
\begin{align}
(\Psi_\epsilon^{k}(z^0;\tau))_i&\leq(\Psi_\epsilon^{k-1}(z^0;\tau))_i-\frac 1{1+\epsilon}\partial_i F\left( \Psi_\epsilon^{k_j}(z^0;\tau)\right)\label{approxtoshowsigns1}\\
(\Psi_\epsilon^{k}(z^0;\tau))_i&\geq(\Psi_\epsilon^{k-1}(z^0;\tau))_i-\frac 1{1-\epsilon}\partial_i F\left( \Psi_\epsilon^{k_j}(z^0;\tau)\right).\label{approxtoshowsigns2}
\end{align}
These last estimates easily follow from the first requirement in \eqref{rdef}, i.e., 
$$\mid(\Psi_\epsilon^{k}(z^0;\tau))_i-\Psi_\epsilon^{k-1}(z^0;\tau))_i+\tau \partial_i F\left( \Psi_\epsilon^{k}(z^0;\tau)\right)\mid\leq \tau \epsilon\mid(\Psi_\epsilon^{k}(z^0;\tau))_i-\Psi_\epsilon^{k-1}(z^0;\tau))_i \mid,$$
under the assumption 
$$\sign \partial_i F\left( \Psi_\epsilon^{k}(z^0;\tau)\right)=-\sign (\Psi_\epsilon^{k}(z^0;\tau))_i,\;\;\forall k\in \naturals,$$
which is the negative of \eqref{claimsign}. Recall that this part of the proof of Theorem \ref{thm:concistency} is carried out by contradiction.  

We stress that we fully used the stopping criterium of Algorithm \ref{algo1} (i.e., \eqref{rdef}) in this last part of the proof. While the sequence $(z^k)_i=\Psi(z^0;\tau)$ in Theorem \ref{thm:concistency} is shown to have constant sign due to the convexity and symmetry properties of the function $g$, here $(\Psi_\epsilon(z^0;\tau))_i$ has constant sign because this condition is explicitly enforced in \eqref{rdef}.
\end{proof}

Aa (sharp) estimate for the rate of convergence of Algorithm \ref{algo1} depending on the {\L}ojacievicz exponent of $F$ at the limit point follows from Proposition 2.5 of \cite{MePi10}, using the stability estimate of Proposition \ref{prop:stability}, the convergence of  Algorithm \ref{algo1} proven in Theorem \ref{thm:convergence}.
\begin{proposition}[Rate of convergence Algorithm \ref{algo1}]\label{prop:rateofconv}
Let $z^0\in \reals^M$ and let $z^k:=\Psi^{(k)}_\epsilon(z^0;\tau)$ be computed by Algorithm \ref{algo1}, where $\tau$ and $\epsilon$ have been setted accordingly to the hypothesis of Theorem \ref{thm:concistency}, Theorem \ref{thm:convergence}, and Proposition \ref{prop:stability}. Let us denote by $z^*$ the limit of $z^k,$ then we have
\begin{enumerate}[i)]
\item If the {\L}ojaciewicz Inequality \eqref{eq:Loja} holds for $F$ at $z^*$ with $\vartheta<1/2$ and $L>0$,  then there exists $\hat k\in \naturals$ such that, for any $k>\hat k$, we have
\begin{equation}\label{polynomialspeedAlg}
\mid z^k-z^*_h\mid\leq \left((F(z^0)-F(z^*))^{-(1-2\vartheta)}+\frac{1-2\vartheta}{L^2}(k-\hat k)  \right)^{-1/(1-2\vartheta)}.
\end{equation} 
\item If the {\L}ojaciewicz Inequality \eqref{eq:Loja} holds for $F$ at $z^*$ with $\vartheta=1/2$ and $L>0$ (in particular if Hypothesis \ref{hypP3} holds), then there exists $\hat k\in \naturals$ such that, for any $k>\hat k$, we have
\begin{equation}\label{exponentialspeedAlg}
\mid z^k-z^*\mid\leq \mid z^0-z^*\mid \exp\left(-\frac{1}{L^2}(k-\hat k)  \right).
\end{equation}  
\end{enumerate}
\end{proposition}
\begin{proof}[Sketch of the proof] Estimate \eqref{polynomialspeedAlg} with $z^k$ replaced by $\Psi^{(k)}(z^0;\tau)$ holds due to \cite[Prop. 2.5]{MePi10}. Notice that, as it is shown in the proof of Theorem \ref{thm:concistency}, for the values of $\tau$ we are considering, the backward Euler scheme definig $\Psi^{(k)}(z^0;\tau)$ is indeed a variational scheme, i.e., $z^k\in \argmin F(z)+\|z-z^{k-1}\|^2/(2\tau)$. This property is fundamental for applying the results of \cite{MePi10}. Then we can obtain \eqref{polynomialspeedAlg} using the stability estimate of Proposition \ref{prop:stability}. The special case of equation \ref{exponentialspeedAlg} is obtained when the {\L}ojaciewicz exponent $\vartheta$ of $F$ at $z^*$ is $1/2$. This holds in particular when $\Hess F(z^*)$ is non-degenerate.

As it is pointed out in the proof of Proposition \ref{prop:errorbound3}, the Hessian of $F$ at $z^*$ is positive definite. Hence in such a case we have $\vartheta=1/2$. 
\end{proof}

\begin{remark}\label{rem:rateofconveta}
We stress that a result similar to Proposition \ref{prop:rateofconv} can be proven for the minimization of $F_\eta$ instead of $F$, i.e., when Problem \ref{P3} is ill-posed and we consider its regularized version Problem \ref{P4}. In such a case the exponent $\vartheta$ of the {\L}ojaciewicz Inequality is $1/2$ if equation \ref{H4} holds. 
\end{remark}

\subsection{A modified algorithm and its implementation}
Let us recall that the convergence of Algorithm \ref{algo1} proven in Theorem \ref{thm:convergence} depends on the right choice of the parameter $0<\tau<\tau^*$, where the unknown parameter $\tau^*$ depends only on the upper bound $c$ on $F$ computed at the initial guess $z^0$. A carefull examination of the proof of Theorem \ref{thm:concistency} shows that we can pick larger $\tau^*$ as we move our initial guess $z^0$ along the trajectory of the flow emanating from $z^0.$  Also note that a trade off is needed here: larger values of $\tau$ lead to faster convergence of the exact discrete trajectories $\Psi^{k}(z^0)$, but may destroy the convergence of the Newton's Method that we use for approximating $\Psi^{k}(z^0)$ by $\Psi_\epsilon^{k}(z^0).$

A possible way to overcome such difficulty is to apply the following heuristics. Let us pick an initial guess for $\tau^0$ and a maximum number $r_{max}$ of Newton iterations for each time step. If our guess for $\tau^0$ is good, then Newton's Method is conveging quadratically to $z^1$, thus $z^{1,r}$ should meet the stopping criterion of Newton iteration of Algorithm \ref{algo1} for small values of $r<r_{max}$. In such a case we may try to use a larger $\tau^1$ for computing $z^2,$ e.g., $\tau^1=\alpha \tau^0$ with $\alpha>1$. Conversely, if  $z^{1,r_{max}}$ does not meet the stopping criterion, then we reduce the time step by a multiplicative factor $\beta<1$ and \emph{restart} Newton's Method with the previous initial guess. Clearly we need to introduce a maximum number of restarts as well, in order to prevent an infinite loop. 

Iterating the above procedure at each time step we obtain Algorithm \ref{algo2} below.
\footnotesize
\begin{algorithm}[h]
\caption{Compute D-optimal design with adaptive time step choice}\label{algo2}
\begin{algorithmic}
\State{Input $z^0\in \reals^M$: $\nabla F(\Sqm{z^0})\neq 0$, $\tau>0$, $n_{step}\in \naturals$, $toll>0$, $\epsilon>0$, $r_{max}$, $\alpha>1$, $\beta\in(0,1)$, $maxn_{restart}>0$}
\State{Set $k:=0$}
\State{Compute $res=\mid\nabla F(z^0)\mid$}
\If{$res=0$}
\State{\textbf{Exit} with error.}
\EndIf
\State{$z^{old}=z^0$}
\While{$k<n_{step} \vee res>toll$}
\State{Set $k:=k+1$, $z^{new}:=z^{old}$, $n_{restart}:=0$ go:=1}
\State{Compute $res_{Newton}:=\nabla g(z^{new};z^{old},\tau)$  }
\While{go=1}
\State{Set $r:=0$}
\While{$\left(\exists i :\frac{\mid(res_{Newton})_i\mid}{\mid(z^{new}-z^{old})_i\mid}\geq \epsilon  \wedge \sign z^{new}\neq \sign z^{old}\right)\vee r<r_{max}$}
\State{$r=r+1$}
\State{Compute $z^{new}=z^{new}-[\Hess g(z^{new};z^{old},\tau)]^{-1}\nabla g(z^{new};z^{old},\tau) $}
\State{Compute $res_{Newton}:=\nabla g(z^{new};z^{old},\tau)$}
\EndWhile
\If{$\frac{\mid(res_{Newton})_i\mid}{\mid(z^{new}-z^{old})_i\mid}<\epsilon\text{ for all }i  \vee \sign z^{new}= \sign z^{old}$}
\State{go=0, $\tau=\alpha*\tau$}
\Else
\If{$n_{restart}<maxn_{restart}$}
\State{$n_{restart}=n_{restart}+1$,  $\tau=\beta*\tau$}
\Else
\State{Break}
\EndIf
\EndIf
\State{Compute $res=\mid\nabla F(z^{new})\mid$}
\EndWhile
\EndWhile\\
\Return $\Sqm{z^{new}}$
\end{algorithmic}
\end{algorithm}
\normalsize

In order to test the performances of Algorithm \ref{algo2}, we implemented it in matlab language as core rutine of the package \textsc{OptimalDesignComputation}, free downloadable at \url{https://www.math.unipd.it/~fpiazzon/Software/OptimalDesignComputation/}. 

Clearly, the only part of Algorithm \ref{algo2} (and of Algorithm \ref{algo1}) that has a non straightforward implementation is the computation of $\nabla g(z^{new};z^{old},\tau)$ and $\Hess g(z^{new};z^{old},\tau)$ that requires in particular the computation 
\begin{align*}
&\nabla F(z^{new})=\left(1-\frac{B(x_i;\Sqm{z^{new}})}N\right)_{i=1,\dots, M}\\
&\Hess F(z^{new})=\left(\frac{K^2(x_i,x_j;\Sqm{z^{new}})}N\right)_{i,j=1,\dots, M}.
\end{align*} 
Indeed this requires the computation of an orthonormal basis for the linear space generated by the columns of the matrix $V$, where $V_{i,j}=\phi_j(x_i)$, with respect to the scalar product $\langle \phi_h;\phi_k\rangle_{\Sqm{z}}:=\sum_{i=1}^M \phi_h(x_i)\phi_k(x_i)z_i^2.$ 

This task may be accomplished by various techniques that aim to cope with the potential ill-conditioning of such a problem. In the \textsc{OptimalDesignComputation} package this computation is performed by the matlab function \texttt{ONB}, which implements an orthogonalization of the matrix $\diag z V$ by two QR factorization and backslash operator. This tecnique has already been used for the computation of multivariate orthonormal polynomials with polynomial meshes (see e.g., \cite{BoSoVi10,PiVi14}, and \cite{Pi19}). It has been shown that the algorithm is particularly robust, since it can effectively work with Vandermonde matrices with very high condition number, e.g., close to the reciprocal of machine precision \cite{BoDeSoVi11}.  

\subsection{A regularized algorithm for the ill-posed case}
When Problem \ref{P3} is ill-posed (or very ill-conditioned) we can use the machinery we develop so far to solve Problem \ref{P4}, i.e., the regularized version of Problem \ref{P3} that we introduce and study in Subsection \ref{subsecIllPosed}. The estimate \eqref{} suggests that if we solve Problem \ref{P4} for a given value of $\eta:=\eta_n$, i.e., we compute $z_{\eta_n}^*$, then we may try to use this as intitial guess for solving Problem \ref{P4} for a smaller value of $\eta=\eta_{n+1}$. We iterate this procedure, stopping the iteration when $\|z_{\eta_{n+1}}^*-z_{\eta_n}^*\|$ is smaller of a prescribed tollerance.

The design computed in such a way is tipically non-sparse. while for practical applications the sparsity of optimal designs is a very useful property. To overcome such an issue we can use the Caratheodory Tchakaloff compression of a discrete measure (see, e.g., \cite{SoVi15}, \cite{PiSoVi17} and references therein) to compute a design having the same moments on the space $\Phi^2$ but (possibly) much smaller support.

These ideas is summarized in Algorithm \ref{algo3} below, where $\sigma:\realspos\rightarrow \realspos$ is any monotone increasing function. 
\begin{algorithm}[h]
\caption{Compute \emph{non-unique} D-optimal design with CaTch compression}\label{algo3}
\begin{algorithmic}
\State{Input parameters of Algorithm \ref{algo1} (or Algorithm \ref{algo2}), $\eta^0\in \realspos$, $z^0\in \reals^M$: $\nabla F_{\eta^0}(\Sqm{z^0})\neq 0$, $n_{max}^\eta\in \naturals$, $toll>0$}
\State{Set $n=0$, $res=toll+1$}
\While{$n<n^\eta_{max} \vee res>toll$}
\State{Compute $z^{n+1}$ that (approximatively) minimize $F_{\eta_n}$ by Algorithm \ref{algo1} (or Algorithm \ref{algo2}) started at $z^n$}
\State{Set $\eta^{n+1}=\sigma(\eta^{n})$, $res=\|z^{n+1}-z^{n}\|$, $n=n+1$ }
\EndWhile
\State{$V:=V(\Phi^2;X)$, $m:=V^t\Sqm{z^{n}}$}
\If{$Card \support z^n\geq \ddim_X \Phi^2$}
\State{Compute a compressed non-negative measure with weights $w$: $V^tw=m$ by Caratheodory Tchakaloff Theorem either by non-negative least squares or linear programming following e.g. \cite{PiSoVi17}}
\EndIf\\
\Return{$w$}
\end{algorithmic}
\end{algorithm}
\normalsize

\section{Experiments}\label{SecExperiment}
In this section we display the features of Algorithm \ref{algo2} and test the performances of its implementation (which is the core rutine of the aforementioned \textsc{OptimalDesignComputation} package) on few test cases of relatvely small dimension. We consider only examples where the statistical model is of polynomial type, i.e., $\Phi$ is some polynomial space. We remark that this is done only for practical reasons, there is no limitation for the choice of the basis functions in the \textsc{OptimalDesignComputation} package.

As a first example we compare the performances of Algorithm \ref{algo2} with the one of Algorithm \ref{algo1}.
\begin{experiment}[Chebyshev-Lobatto Grid]\label{exp1}
Let $X$ be a degree $40$ Chebyshev-Lobatto grid, i.e., the cartesian product of $41$ by $21$ Chebyshev-Lobatto points in $[-1,1]$, so $M=1682$. We pick $\Phi$ as the space of polynomials with total degree not exceeding $4$ (hence $N=15$). Let $z^0=1/M(1,1,\dots,1)^t)$. Consider the two parameters setting of Algorithm \ref{algo2}:
\begin{enumerate}[a)]
\item $\alpha=\beta=1, \tau^0=1$
\item $\alpha=\beta=1.15$, $\tau^0=1$, $\epsilon=10^{-4}$, and $r_{max}=5$.
\end{enumerate}
\end{experiment}
Note that in the case \emph{a)} of Experiment \ref{exp1} we set $\alpha=1$ to run Algorithm \ref{algo1} using the implementation of Algorithm \ref{algo2}. 

In the two experiments essentialy the same design is computed, e.g. the computed weights agree up to $10^{-15}.$ The common design support is reported in Figure \ref{squaredesignsupport}. Note that the cardinality of the support of the optimal design is $25$ which lies in the admissible interval for the cadinality of an optimal design $[\ddim_X \Phi,\ddim_X \Phi^2]=[15,45].$
\begin{figure}[h]
\begin{center}
\caption{The support (circles) of the optimal design of Experiment \ref{exp1}.}
\label{squaredesignsupport}
\includegraphics[scale=0.7]{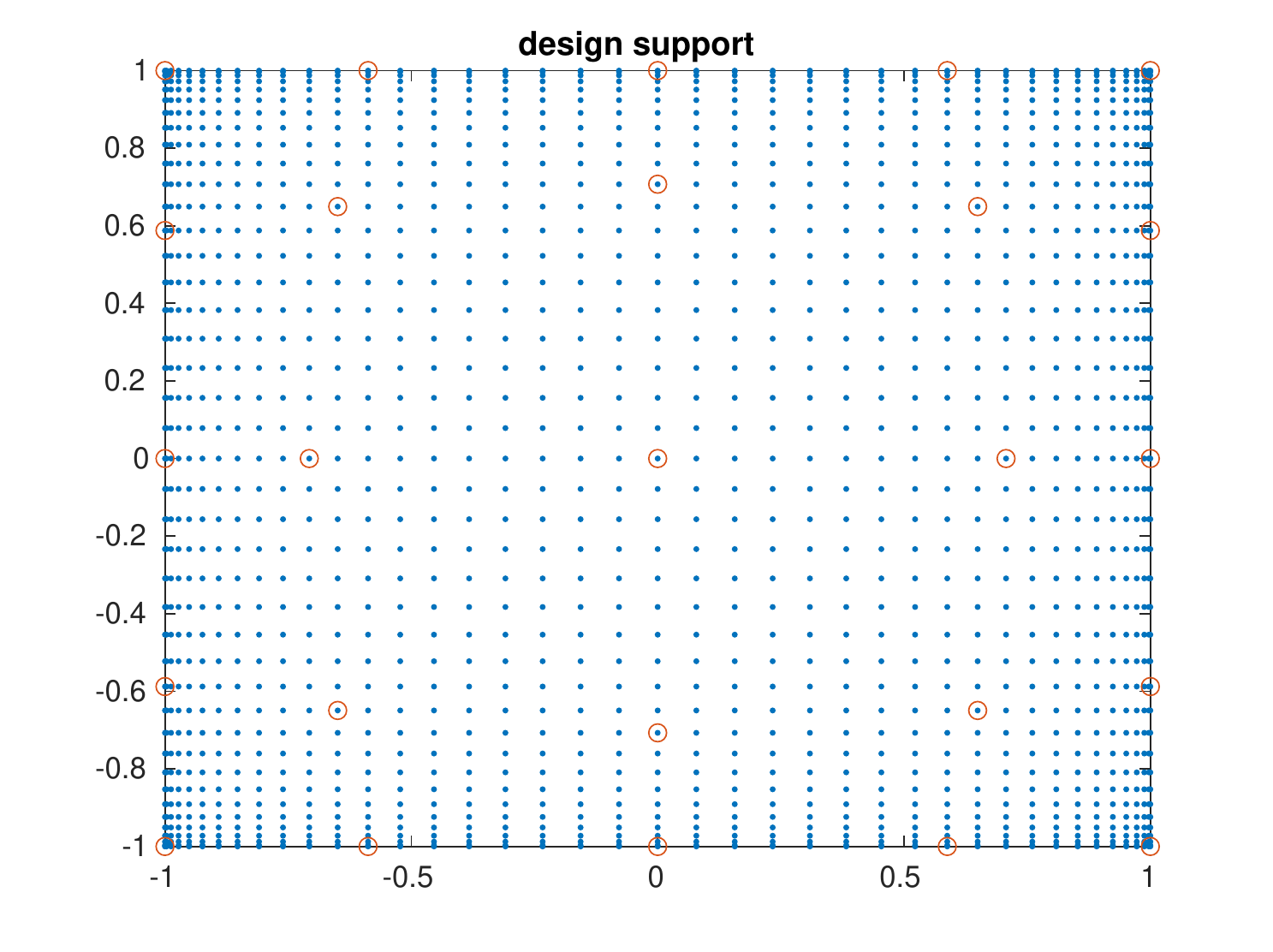}
\end{center}
\end{figure}
The computed design can be termed optimal, since it meets the Karush Kuhn Tucker optimality conditions \eqref{kktP2} up to machine precision, as we report in Figure \ref{squaredesigequilibrium}.

\begin{figure}[h]
\begin{center}
\caption{Check of the Karush Kuhn Tucker optimality conditions of the candidate optimal design of Experiment \ref{exp1} (both case \emph{a)} and \emph{b)}). Upper graph shows that $\partial_i E(w^*)$ is non-negative for any $i$ such that $w^*_i=0$, while the lower graph reports $\mid\partial_i E(w^*)\mid$ for $i$ such that $w^*_i\neq 0$.}
\label{squaredesigequilibrium}
\includegraphics[scale=0.55]{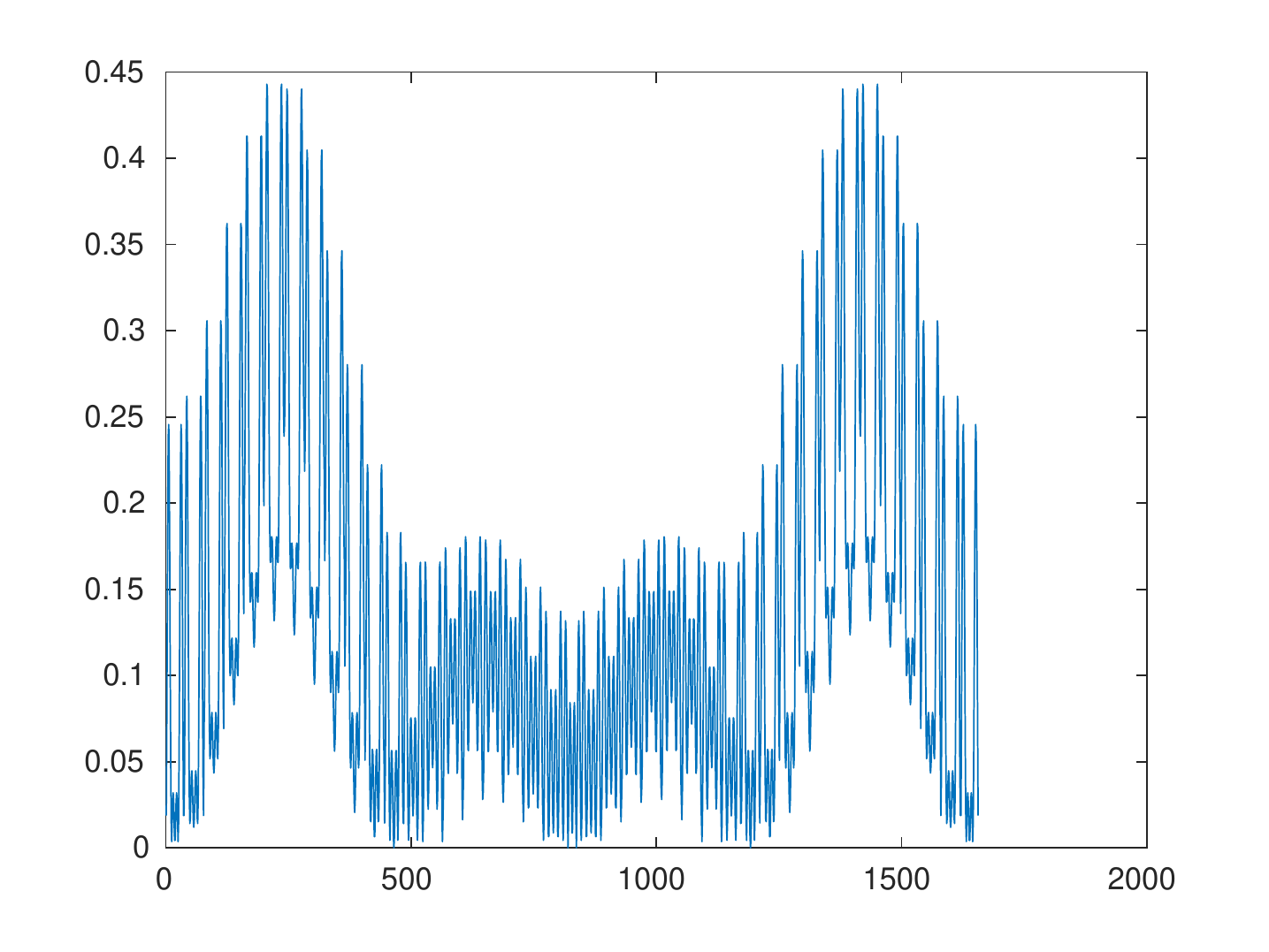}\\
\includegraphics[scale=0.55]{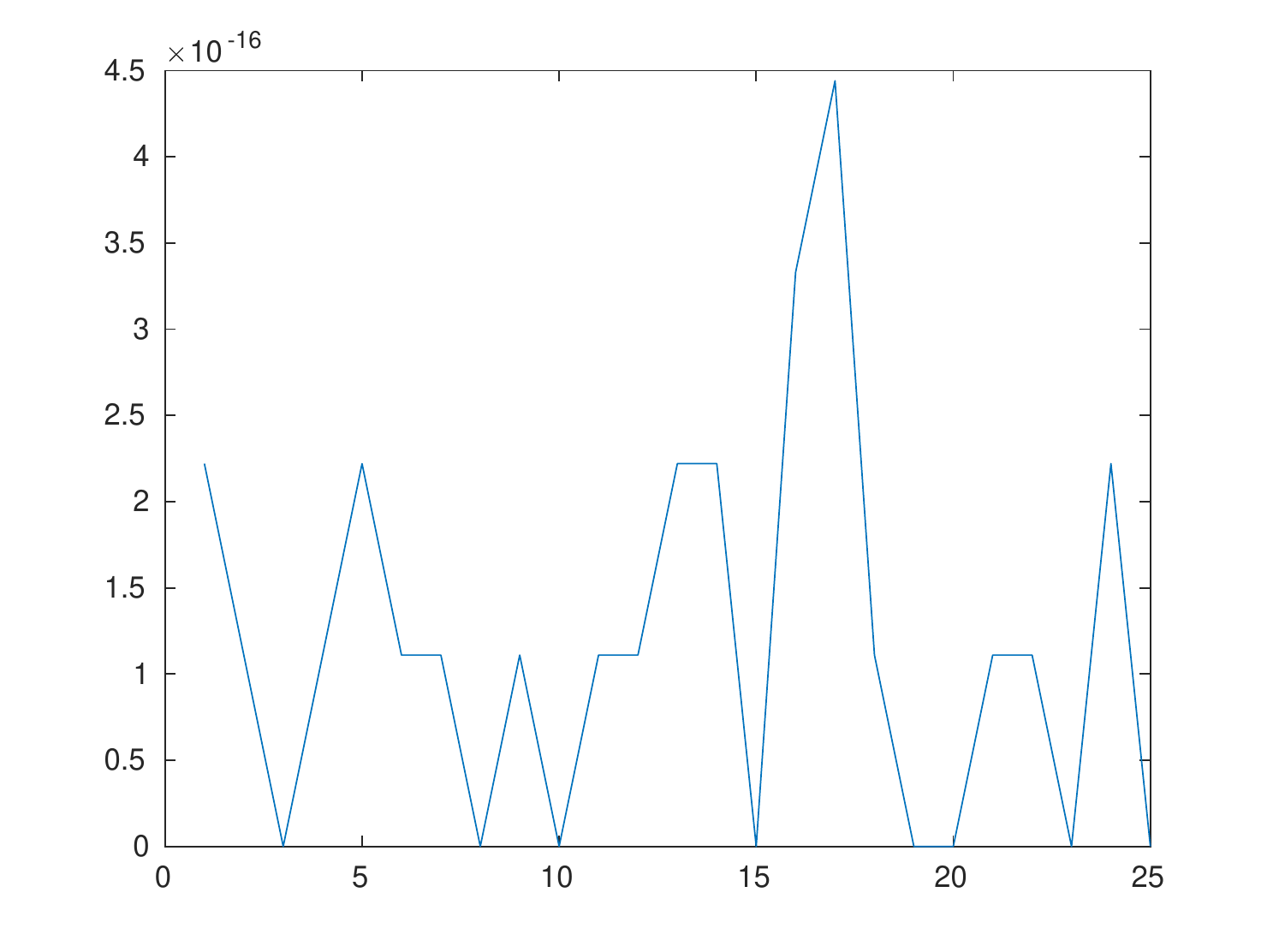}
\end{center}
\end{figure}
The two experiments have a very different experimental convergence behaviour. Indeed, in the case of constant time step $\tau\equiv 1$, the profile of convergence exhibits a linear behaviour, as it is clear from Figure \ref{square1conv}, where the steps, the residual $\|\nabla F(z^k)\|_\infty$, and the Karush Kuhn Tucker residual of the $k$-th iteration are displayed. We remark that we term \emph{Karush Kuhn Tucker residual} the max-norm  of the non-linear residual of the system \eqref{kktP2} (right hand side), i.e., the quantity $\|res^{KKT}(\Sqm{z^k})\|_\infty$, where
$$res_i^{KKT}(w)= 
\begin{cases}
\mid\partial_i E(w)\mid=\mid 1-B(x_i;w)/N\mid& \text{ if }w_i>0\\
\mid(\partial_i E(w))^-\mid=\max\{0,B(x_i;w)/N-1\}& \text{ if }w_i=0
\end{cases}\,.
$$
Instead, enabeling the adaptive time step choice in Algorithm \ref{algo2}, the profile of convergence has a superlinear behaviour, see Figure \ref{squarevarconv}.
\begin{landscape}
\begin{figure}[h]
\begin{center}
\caption{Convergence profile for the computation of the optimal design of Experiment \ref{exp1} (case \emph{a)}) with fixed time step $\tau\equiv 1.$ Steps $\|z^{k+1}-z^k\|$ are plotted on the left axis, the $\ell^\infty $ norms of the residuals and residuals of KKT conditions, i.e, $\|\nabla F(z^k)\|_\infty$ and $\|res^{KKT}(\Sqm{z^k})\|_\infty$ are reported on the right axis. }
\label{square1conv}
\includegraphics[scale=0.45]{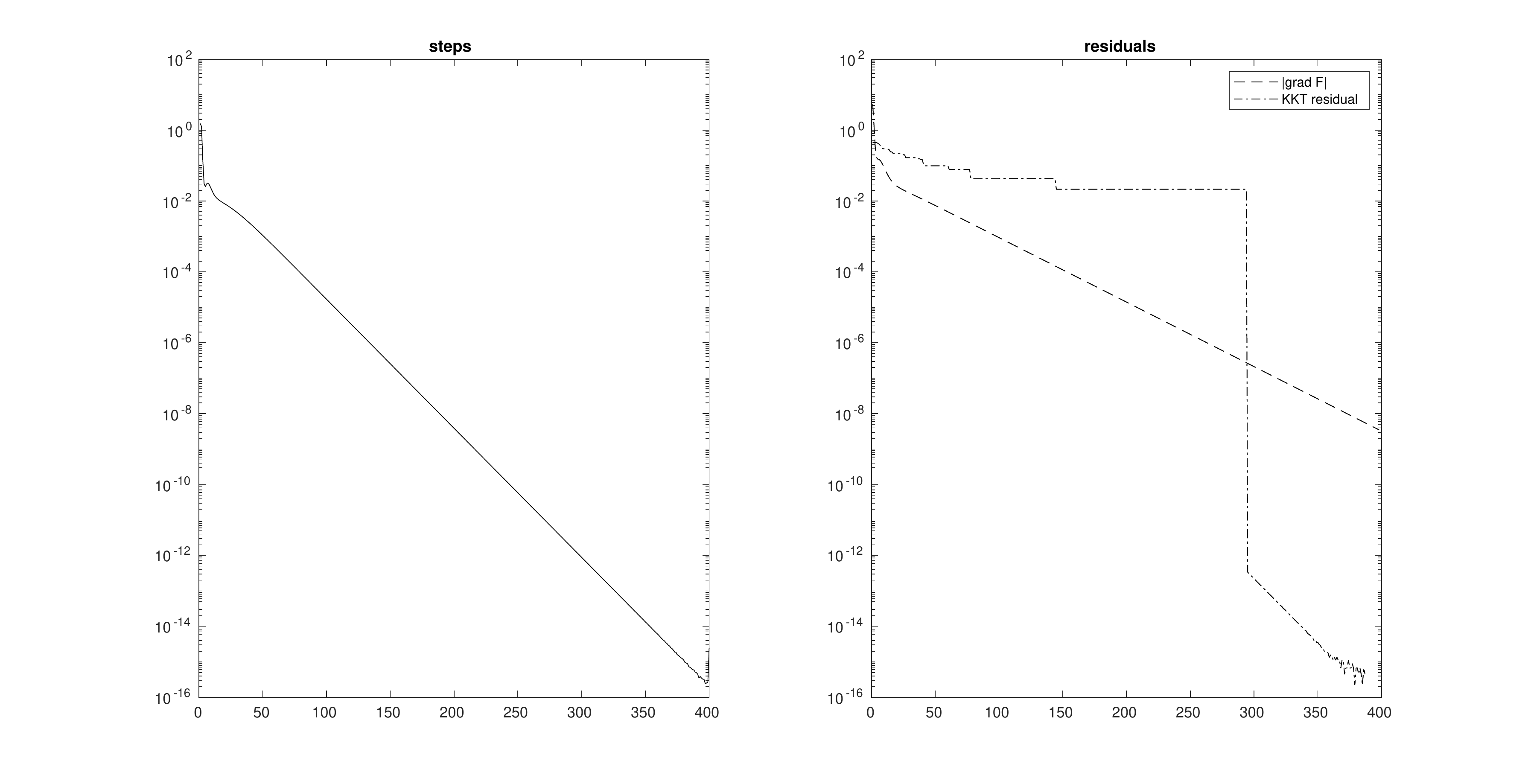}
\end{center}
\end{figure}
\end{landscape}

\begin{landscape}
\begin{figure}[h]
\begin{center}
\caption{Convergence profile for the computation of the optimal design of Experiment \ref{exp1} case \emph{b)} with variable time step $\alpha=\beta=1.15$. Steps $\|z^{k+1}-z^k\|$ are plotted on the left axis, the $\ell^\infty $ norms of the residuals and residuals of KKT conditions, i.e, $\|\nabla F(z^k)\|_\infty$ and $\|res^{KKT}(\Sqm{z^k})\|_\infty$ are reported on the right axis. }
\label{squarevarconv}
\includegraphics[scale=0.45]{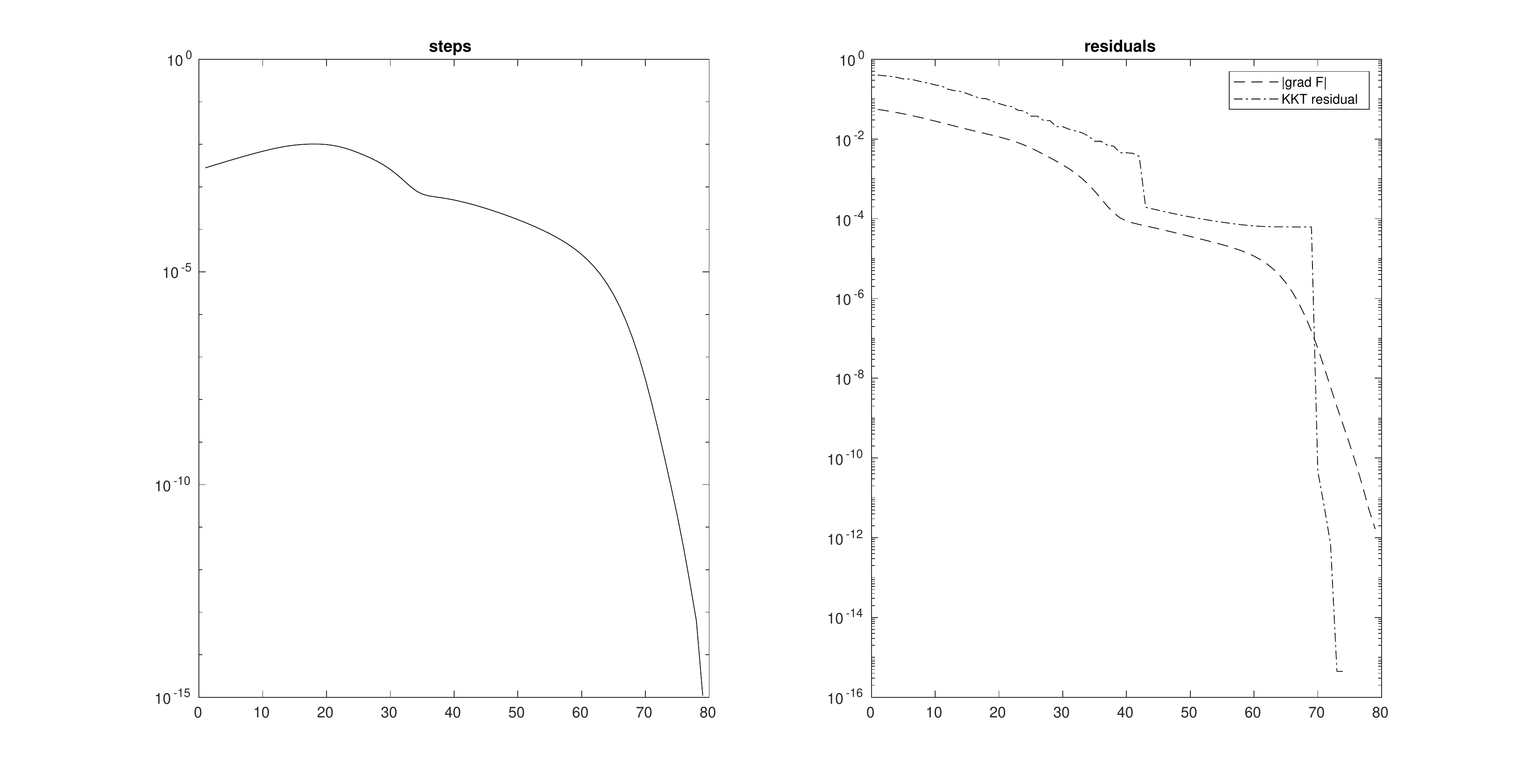}
\end{center}
\end{figure}
\end{landscape}

As second example we consider a rather usual setting in random sampling, a uniform random points cloud in the square $[-1,1]$.
\begin{experiment}[Uniform points cloud]\label{exp2}
Let $X$ be a random points cloud of $M=1600$ uniform points in $[-1,1]$. We pick $\Phi$ as the space of polynomials with total degree not exceeding $10$ (hence $N=66$). Let $z^0=1/M(1,1,\dots,1)^t)$.  Consider the parameters setting of Algorithm \ref{algo2}: $\alpha=\beta=1.15$, $\tau^0=1$, $\epsilon=10^{-4}$, and $r_{max}=5$.
\end{experiment}
Also for Experiment \ref{exp2} we compute an optimal design up to machine precision in the sense of the sense of Karush Kuhn Tucker residual is approximately $10^{-15}$. We report in Figure \ref{uniformdesigequilibrium} the components of the vector of residuals The cardinality of the support is $171$ which again lies in the admissible interval $[66,231].$
\begin{figure}[h]
\begin{center}
\caption{The support (circles) of the optimal design of Experiment \ref{exp2}.}
\label{squaredesignsupport}
\includegraphics[scale=0.7]{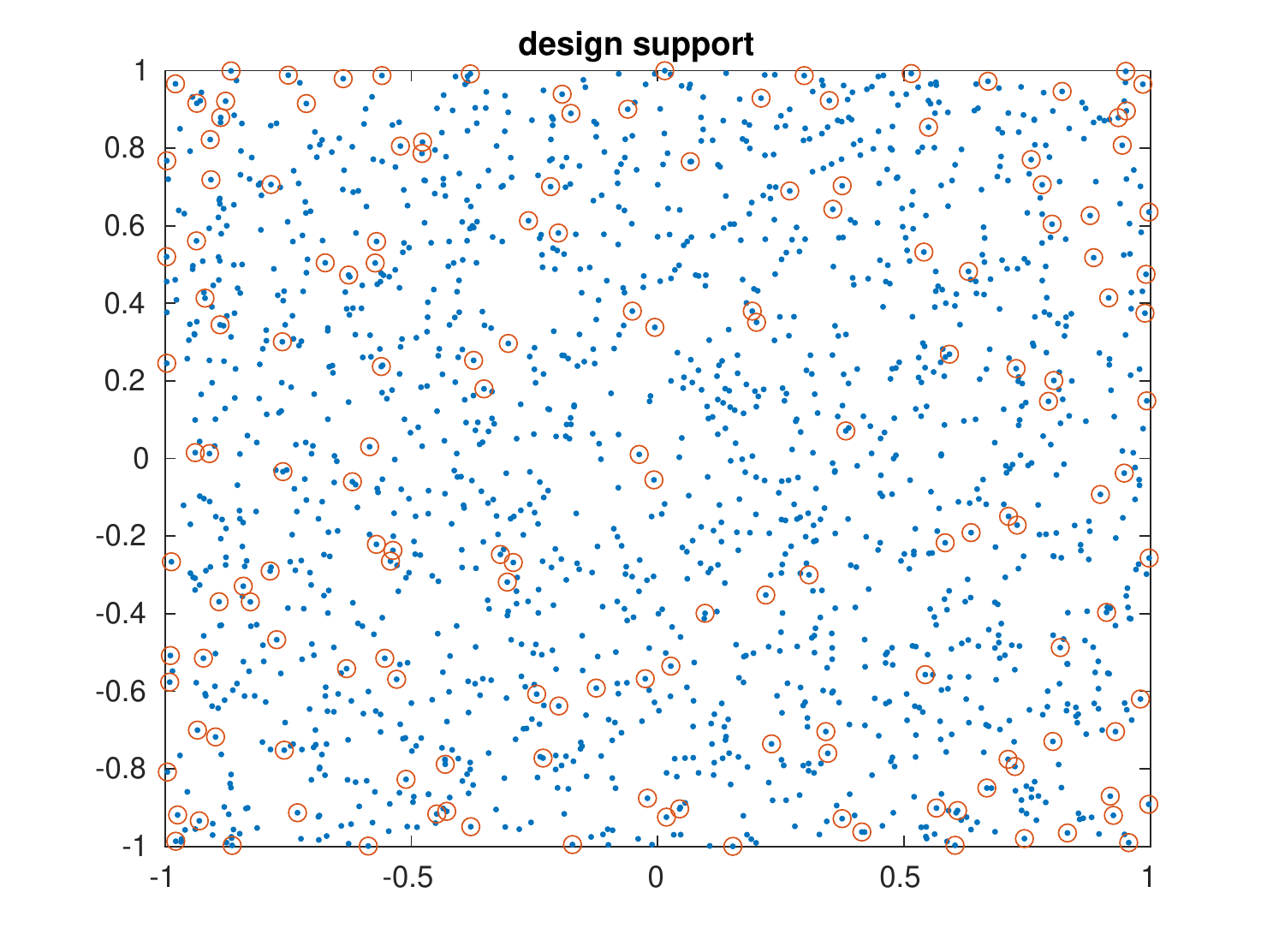}
\end{center}
\end{figure}
\begin{figure}[h]
\begin{center}
\caption{Check of the Karush Kuhn Tucker optimality conditions \eqref{kktP2} (right system) of the candidate optimal design of Experiment \ref{exp2}. Upper graph shows that $\partial_i E(w^*)$ is non-negative for any $i$ such that $w^*_i=0$, while the lower graph reports $\mid\partial_i E(w^*)\mid$ for $i$ such that $w^*_i\neq 0$.}
\label{uniformdesigequilibrium}
\includegraphics[scale=0.55]{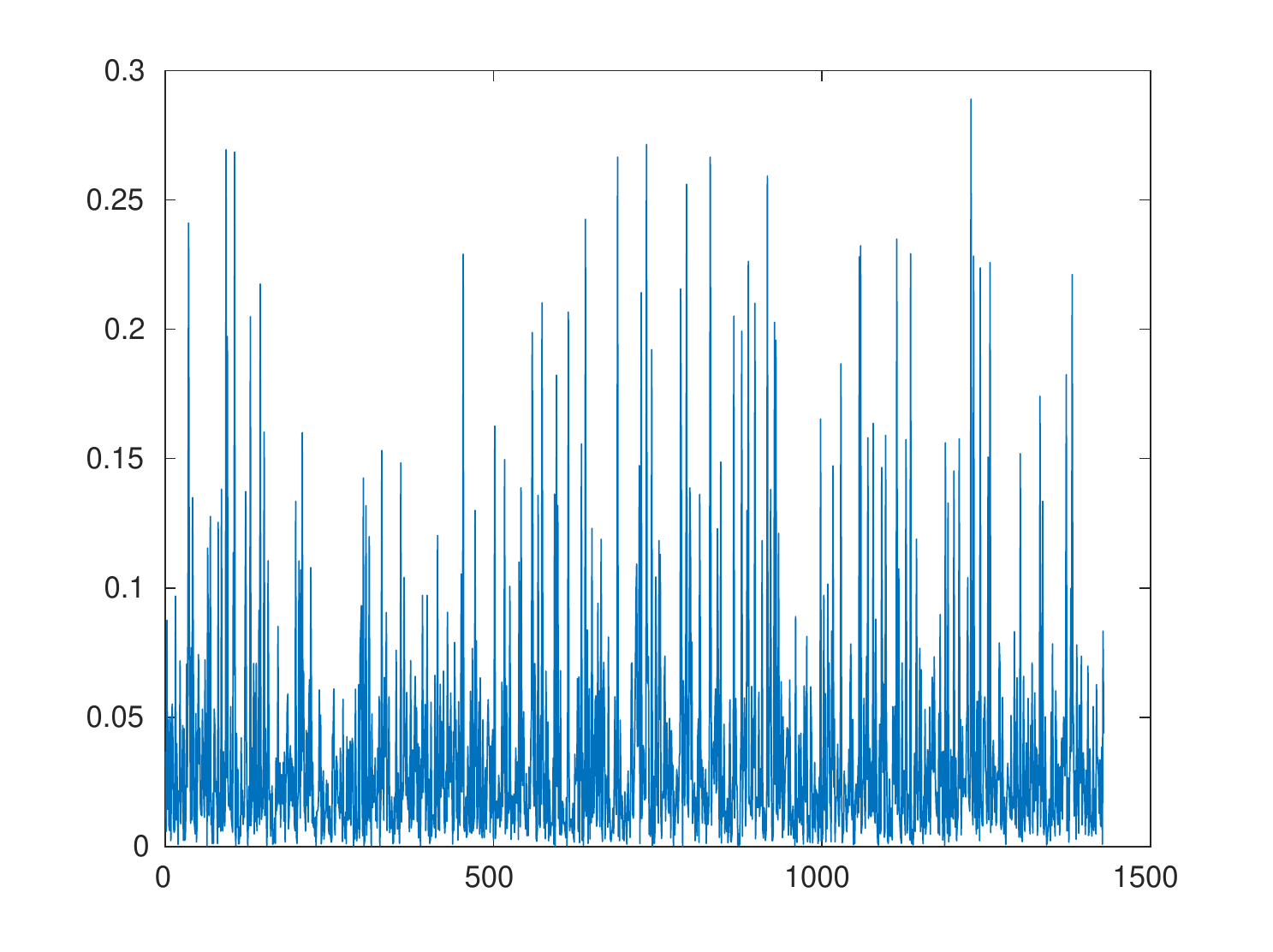}\\
\includegraphics[scale=0.55]{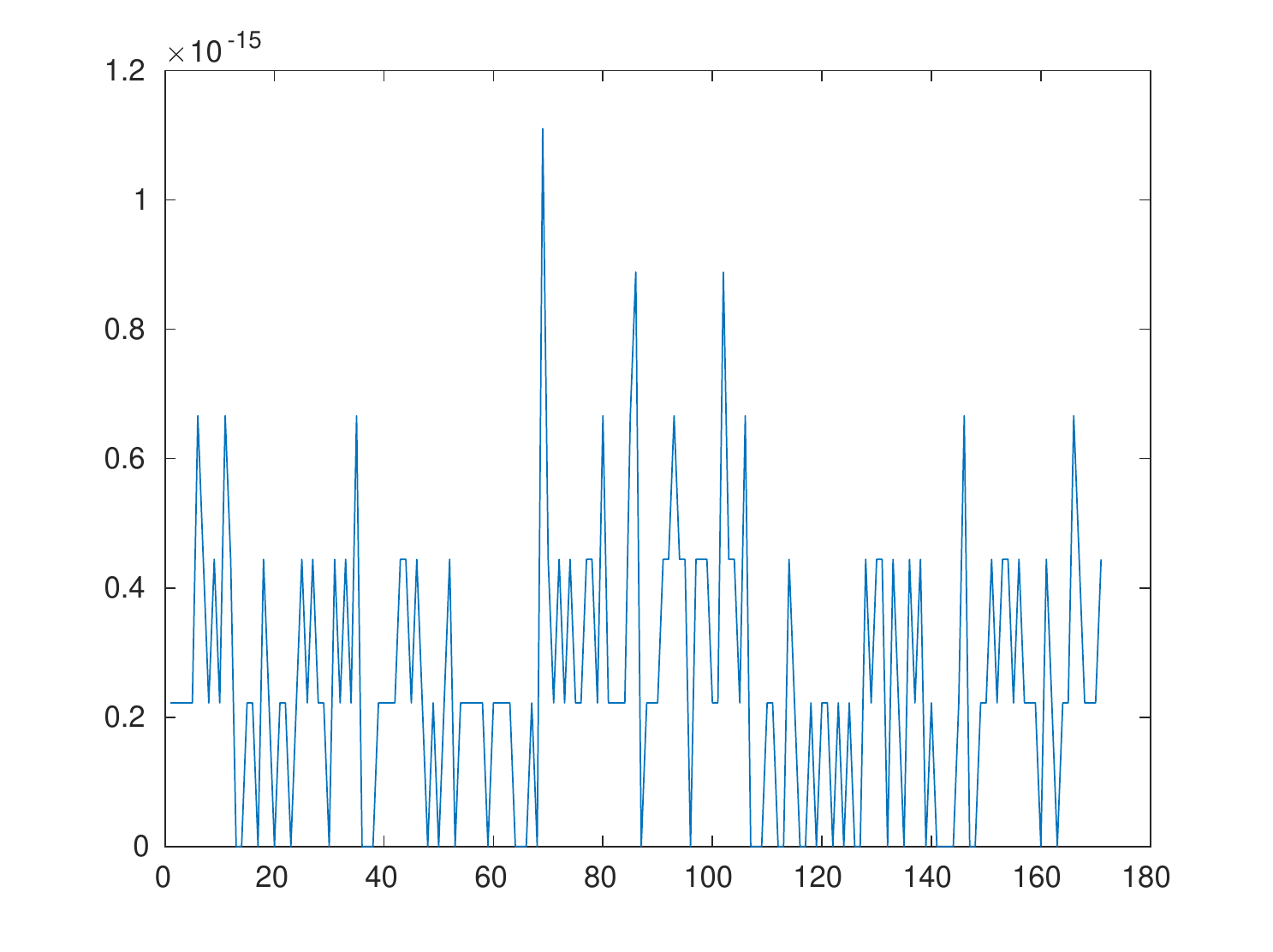}
\end{center}
\end{figure}

\begin{landscape}
\begin{figure}[h]
\begin{center}
\caption{Convergence profile for the computation of the optimal design of Experiment \ref{exp2} with variable time step $\alpha=\beta=1.15$. Steps $\|z^{k+1}-z^k\|$ are plotted on the left panel, the $\ell^\infty $ norms of the residuals and residuals of KKT conditions, i.e, $\|\nabla F(z^k)\|_\infty$ and $\|res^{KKT}(\Sqm{z^k})\|_\infty$ are reported on the right panel.}
\label{uniformconv}
\includegraphics[scale=0.45]{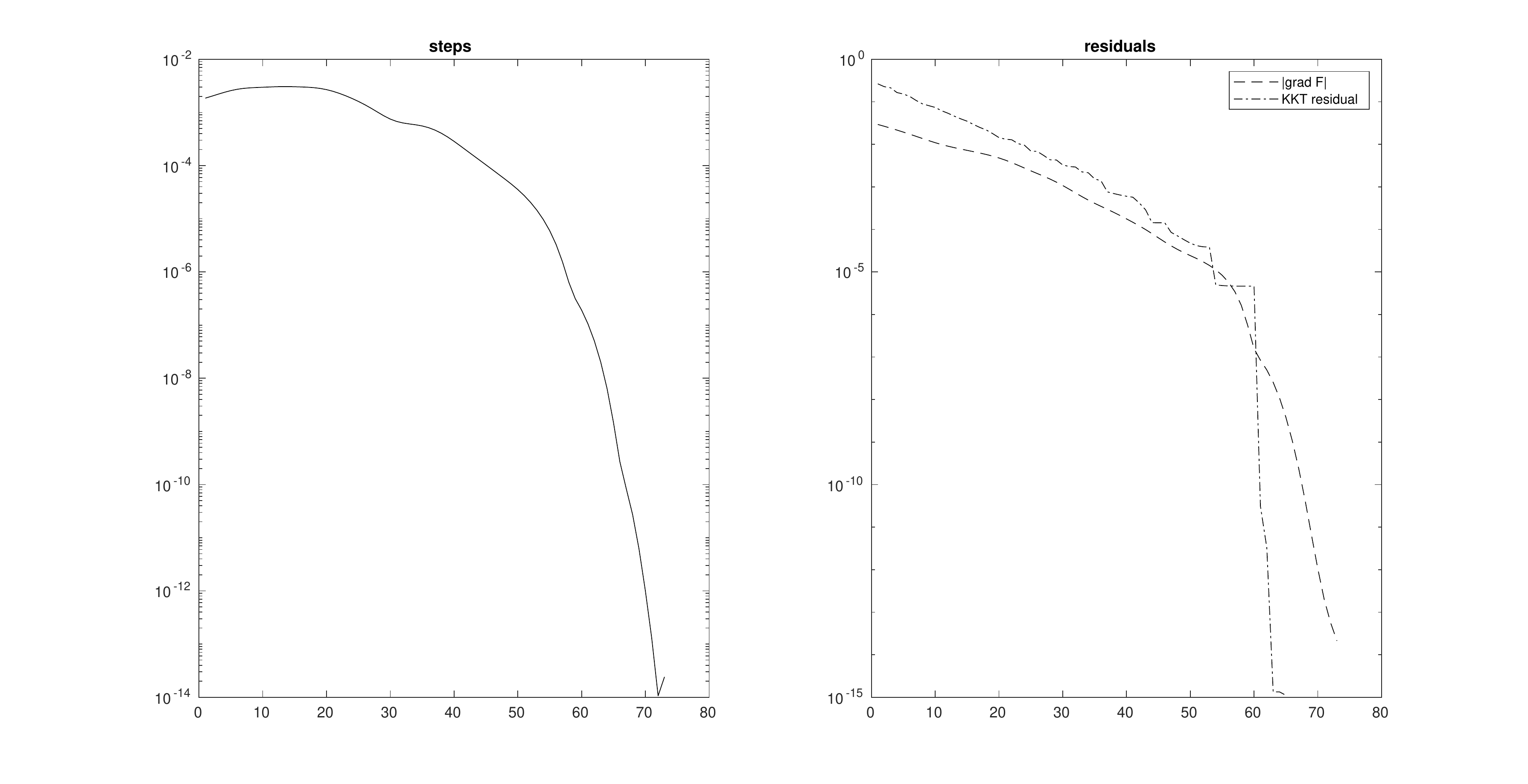}
\end{center}
\end{figure}
\end{landscape}

We report in Figure \ref{uniformconv} the convergence profile of Experiment \ref{exp2}. Note that, both in Experiment \ref{exp1} (case \emph{a)} and case \emph{b)}) and in Experiment \ref{exp2}, the non linear residuals $\| \nabla F(z^{k})\|$ (right panel of Figure \ref{square1conv}, Figure \ref{squarevarconv}, and Figure \ref{uniformconv}) have the same qualitative behaviour of the steps (left panel of Figure \ref{square1conv}, Figure \ref{squarevarconv}, and Figure \ref{uniformconv}). We can check a-posteriori that this quantities are good estimators of the error $\|z^*-z^{k}\|$. Indeed, we can compute $\nabla E$ and $\Hess E$ at our best (e.g., last) approximation $w^{output}$ of $w^*$ and check numerically that Equations \eqref{SAp31} and \eqref{SAp32} hold at $w^{output}$, i.e., the Set of Assumptions \ref{hypP3} holds true. We proved in Proposition \ref{prop:errorbound3} that under such a condition the non-linear residual is proportional to the error, see equation \eqref{eq:linearcond}. 

We stress that this numerical check is rather delicate: we need to distinguish very small values from zero, both in determining the design support and in computing the eigenvalues of the restricted Hessian matrix of $E$ at $w^{output}$ (see Equation \eqref{SAp32}). In critical cases it might be more safe, though more expensive, to directly compute the smallest eigenvalue of the Hessian of $F.$ 

In our tests of the implementation of Algorithm \ref{algo1} and Algorithm \ref{algo2} we tried to construct examples of an optimal design for a finite set for which the Set of Assumptions \ref{hypP3} fails, but still the Set of Assumptions \ref{H2} holds true. This would lead to an example of unique optimal design which is a minimizer of $F$ with degenerate Hessian matrix. Surprisingly, this task is in pactice much more difficult than it could seem at first sight. Unfortunately we are not able to provide a \emph{neat} example of such a critical case. 

Conversely, we can provide an example where even the rather weak Set of Assumptions \ref{H2} does not hold. A very large (and possibly symmetric) design space with respect to the dimension of $\Phi$ is used to contruct the following example. The design space here is a \emph{admissible polynomial mesh} for a disk. Admissible polynomial meshes are good discretizations of a compact set for constructing discrete polynomial lesat squares projection operators having small norm and allowing stable computations. See for instance \cite{BoDeSoVi11}, \cite{PiVi14}, and \cite{BoCaLeSoVi11}.

\begin{experiment}[Admissible mesh for a disk]\label{exp3}
Let $X$ be the admissible polynomial mesh of degree $20$ for the unit disk as it is constructed in \cite[Sec. 2.1]{BoCaLeSoVi11}. We pick $\Phi$ equal to the space of polynomials of two variables with total degree at most $2$. Thus $N=6$ and $M=1601$. Let $z^0=1/M(1,1,\dots,1)^t)$.  Consider the parameters setting of Algorithm \ref{algo2}: $\alpha=\beta=1.15$, $\tau^0=1$, $\epsilon=10^{-4}$, and $r_{max}=5$.
\end{experiment}

\begin{landscape}
\begin{figure}[h]
\begin{center}
\caption{Convergence profile for the computation of the optimal design of Experiment \ref{exp3} with variable time step $\alpha=\beta=1.15$. Steps $\|z^{k+1}-z^k\|$ are plotted on the left panel, the $\ell^\infty $ norms of the residuals and residuals of KKT conditions, i.e, $\|\nabla F(z^k)\|_\infty$ and $\|res^{KKT}(\Sqm{z^k})\|_\infty$ are reported on the right panel.}
\label{diskconv}
\includegraphics[scale=0.45]{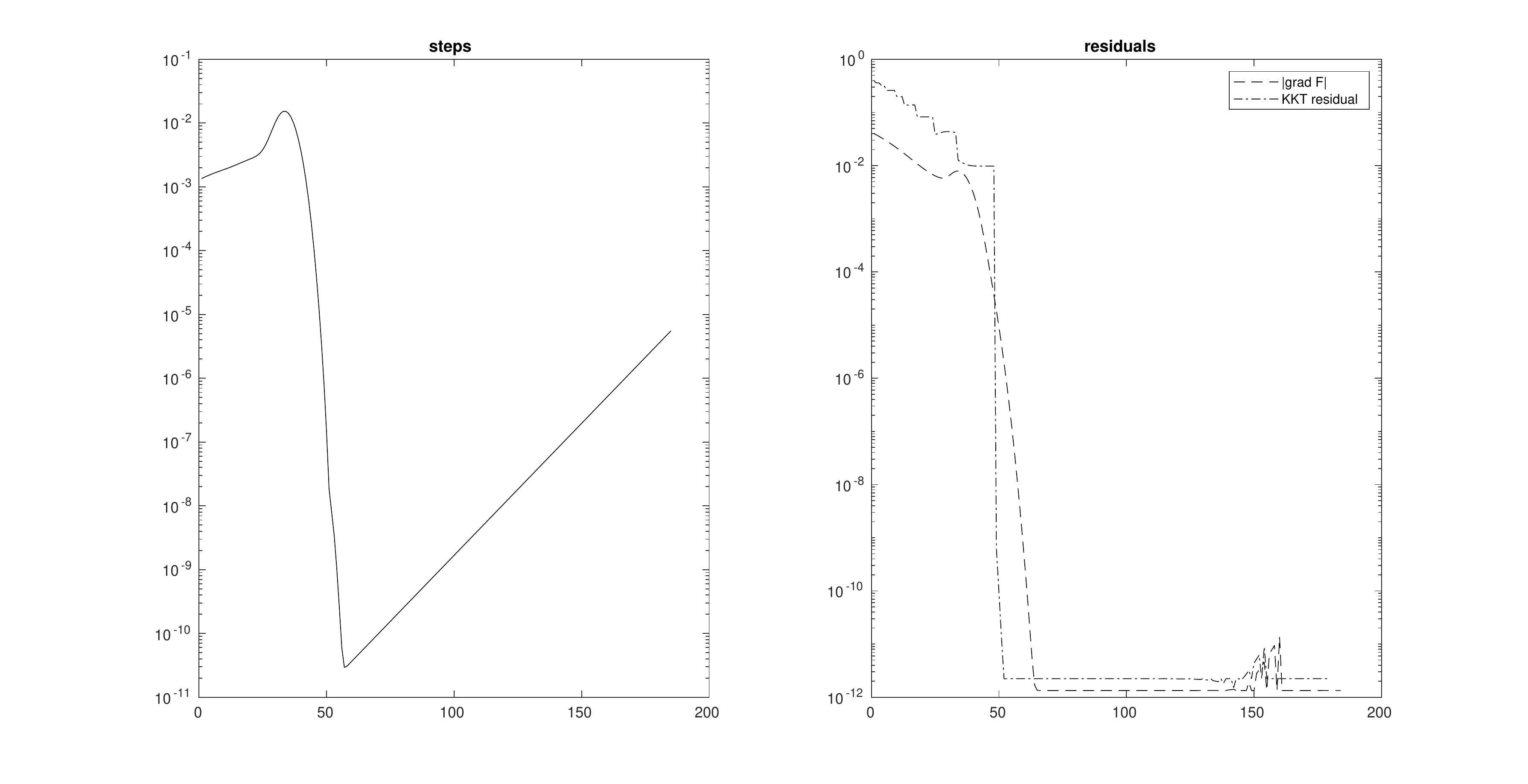}
\end{center}
\end{figure}
\end{landscape}

\begin{figure}
\caption{Spectrum of the Hessian matrix of $F$ computed at the final step of Algorithm \ref{algo2} in Example \ref{exp3}.}\label{diskspectrum}
\begin{center}
\includegraphics[scale=0.5]{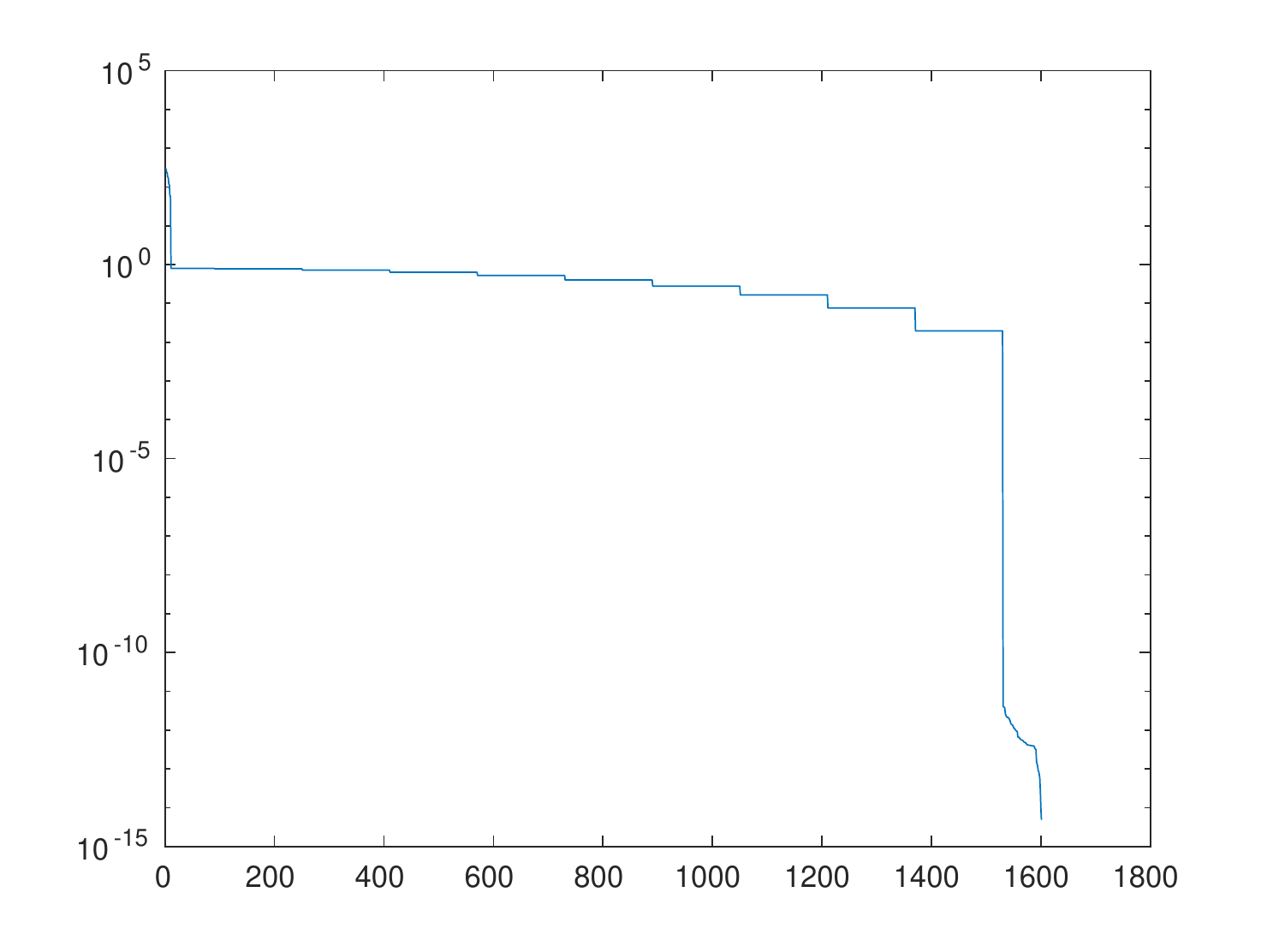}
\end{center}
\end{figure}
We report the steps and the residuals of Experiment \ref{exp3} in Figure \ref{diskconv}. It is clear that this case the qualitative behaviour of the computed sequence is different from the  previous cases. Note that once a "very good accuracy" both in terms of residual and KKT residual is reached, then the iterates of Algorithm \ref{algo2} move along a path with almost constant residual $\|\nabla F(z^{k})\|$ and Karush Kuhn Tucker residual. Since the Newton's method is reaching the stopping criterion very easily (e.g., 1-3 iterations) at each time step, the variable $\tau$ Algorithm \ref{algo2} is increased. This results in a increasing step. The explanation of this phenomena possibly comes from the analysis of the spectrum of the Hessian matrix of $F$ (see Figure \ref{diskspectrum}) which has some very small (i.e., close to machine precision) eigenvalues. We can say that from the numerical point of view the Set of Assumptions \ref{hypP3} does not hold. On the other hand, we may ask wether the Set of Assumptions \ref{H2} is satisfied. To this aim let us denote by $w^{\text{output}}$ our smallest residual approximation of an optimal design and consider the following problem:
\begin{equation}\label{modifiedunicity}
\begin{cases}
\bar w\in \argmin \log \frac 1{1+\|w-w^{\text{output}}\|^2},\\
\bar w_i\geq 0& \text{ for all }i\\
V^t\bar w=V^tw^{\text{output}}
\end{cases}\;,
\end{equation}
where $V:=V(\Phi^2,X).$  This corresponds to find the point $\bar w$ of the solution set $\mathcal S$ (see Subsection \eqref{subsecIllPosed}) that has the largest distance from $w^{\text{output}}.$ Clearly, if we find a feasible solution to \eqref{modifiedunicity} which is not equal to $w^{\text{output}}$, then the Set of Assumptions \ref{H2} does not hold. Note that we consider the problem \eqref{modifiedunicity} as a convenient way of checking such an hypothesis. This approach has some advantages since it is a classical optimization problem for a smooth and strictly convex function on a politope (or an empty set). In our specific case, we solved \eqref{modifiedunicity} numerically by using the matlab function \texttt{fmincon}, which computes an optimal feasible solution differing from $w^{\text{output}}$ by approximately $7\cdot 10^{-2}.$ Thus the Set of Assumptions \ref{H2} does not hold. Thus the problem we are considering is not well-posed.

The reasons explained above suggest to test Algorithm \ref{algo3} on the same example as in Experiment \ref{exp3}.
\begin{experiment}\label{exp4}
Let $X,\Phi$ be as in Experiment \ref{exp3}. Let $\alpha=\beta=1.5$ and let the other parameters of Algorithm \ref{algo2} be setted as in Experiment \ref{exp3}. Let $z^0=1/M(1,1,\dots,1)^t)$.  Consider Algorithm \ref{algo3} (calling Algorithm \ref{algo2}) with $\eta^0=10^{-2}$ $\sigma(\eta)=\eta^2.$
\end{experiment}
Note that we decided to run Algorithm \ref{algo3} in Experiment  \ref{exp4} with a larger value of $\alpha$ with respect to the case of Experiment \ref{exp3}. This is heuristically justified by the fact that we know that the objective considered in Experiment \ref{exp4}, i.e., $E_\eta$, is strongly convex on $\realspos^M.$

In this example the \texttt{while} loop used in Algorithm \ref{algo3} for diminuishing $\eta$ is stopped after the first iteration since the computed minimizers $z^*_{\eta^0},z^*_{\eta^1}$ are very close. The behaviour of the computed sequences $\{z^k_{\eta^0}\},\{z^k_{\eta^1}\}$ is also very similar, but much different from the one of the sequence $\{z^k\}$ computed by Algorithm \ref{algo2} in Experiment \ref{exp3}. We report the convergence profile of $\{z^k_{\eta^1}\}$ in the left panel of Figure \ref{diskregconv} and the residuals in the right panel of the same figure (we invite the reader to compare this figure to Figure \ref{diskconv}). The experimental convergence is clearly super-linear. This is a consequence of the combination of the choice calling Algorithm \ref{algo2} instead of Algorithm \ref{algo1} in the \texttt{while} loop of Algorithm \ref{algo3}, and the fact that \eqref{H4} holds true (as we can easily checked numerically), see Remark \ref{rem:rateofconveta}. 

We remark also that the computed design is indeed an optimal design instead of just a minimizer of $E_\eta$. This happens (for all $\eta$ sufficiently small) precisely when property \eqref{orthmagic} is satisfied. 

An interesting feature of this example is the cardinalities of the support of the computed designs, see Figure \ref{disksupport}. Indeed if we force Algorithm \ref{algo3} to skip the compression final step, we compute a design supported at $81$ points, while enabeling the compression of the design by Caratheodory Tchakaloff Theorem as implemented in \cite{Vi16,PiSoVi17} the cardinality of the support drops considerably to $10.$ This is a remarkable fact, since in the compression procedure the fitting of $15$ moments (i.e., the dimension of $\Phi^2$) is imposed. In general (and in the large majority of the test we made on the implementation of the Caratheodory Tchkaloff compression, see also \cite{PiSoVi17b}) this results in a compressed quadrature formula with a support size of the same magnitude as the number of imposed moments, with the exeption of few instances where the cardinality drops by 1. Here the drop is much larger (in a relative sense). This phenomena is probably related to the fact that we are compressing a quadrature rule for an \emph{optimal} design, which is intrinsically a sparse measure.

If we repeat Experiment \ref{exp3} considering $\Phi$ the space of polynomials of degree at most $4$ instead of $2$ and with an admissible polynomial mesh of degree $40$ ($M=6401$) instead of $20$ we get a similar convergence profile and residuals. On the other hand the result of the compression of the design is even more relevant. In this case the optimal design computed by Algorithm \ref{algo3} before the compression step is $321$ and it drops dramatically to $15$ after compression. Note that this support cardinality meets precisely the lower bound for the cardinality of an optimal design for the considered space, being $\ddim_X\Phi=15,$ see Figure \ref{disksupportb}.  

\begin{figure}[h]
\begin{center}
\caption{The support ($81$ large dots) of the optimal design of Experiment \ref{exp4} before the compression by Caratheodory Tchakaloff Theorem, and after compression ($10$ stars).}
\label{disksupport}
\includegraphics[scale=0.5]{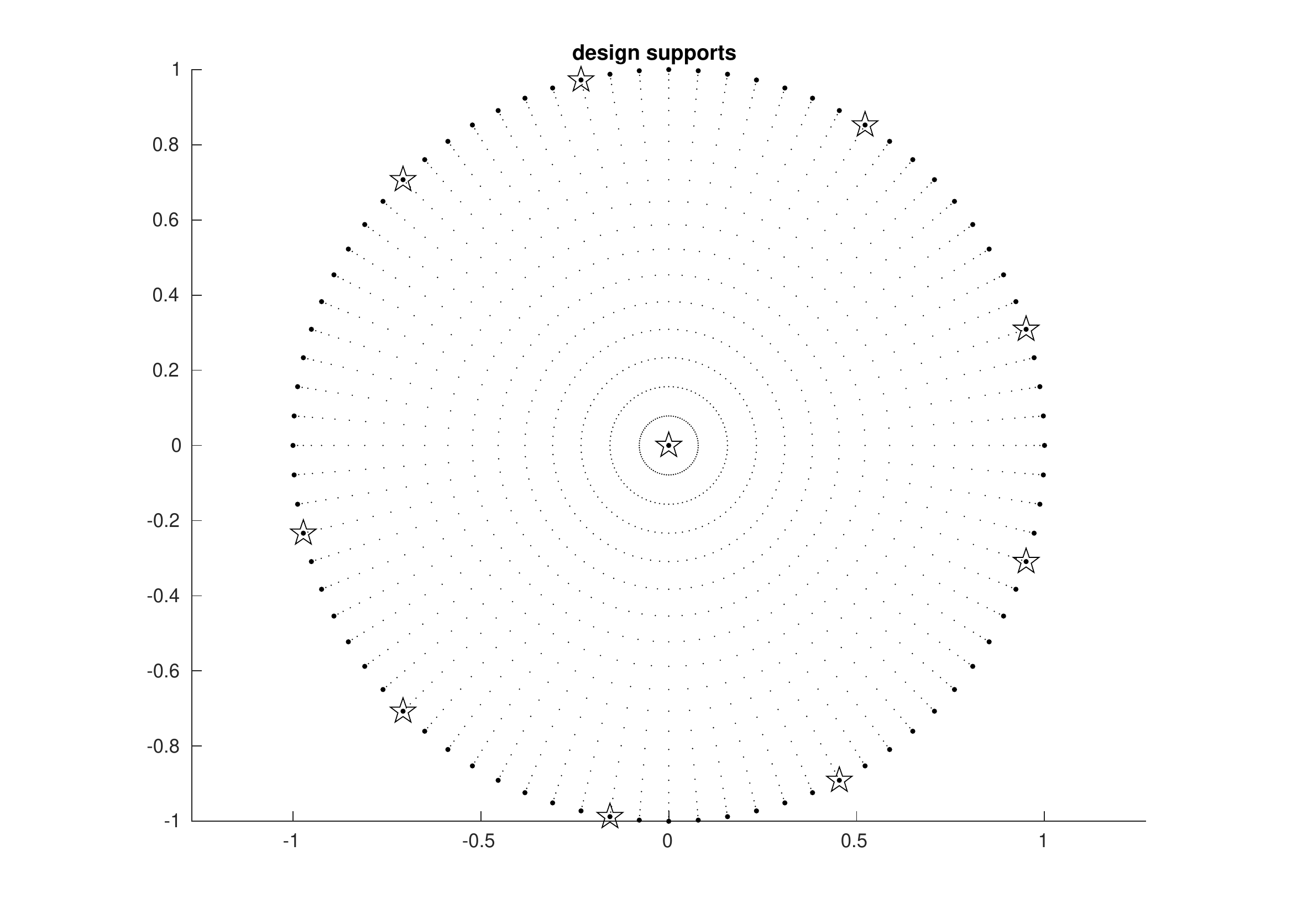}
\end{center}
\end{figure}

\begin{landscape}
\begin{figure}[h]
\begin{center}
\caption{Convergence profile for the computation of the optimal design of Experiment \ref{exp4} with variable time step $\alpha=\beta=1.5$. Steps $\|z^{k+1}-z^k\|$ are plotted on the left panel, the $\ell^\infty $ norms of the residuals and residuals of KKT conditions, i.e, $\|\nabla F(z^k)\|_\infty$ and $\|res^{KKT}(\Sqm{z^k})\|_\infty$ are reported on the right panel.}
\label{diskregconv}
\includegraphics[scale=0.45]{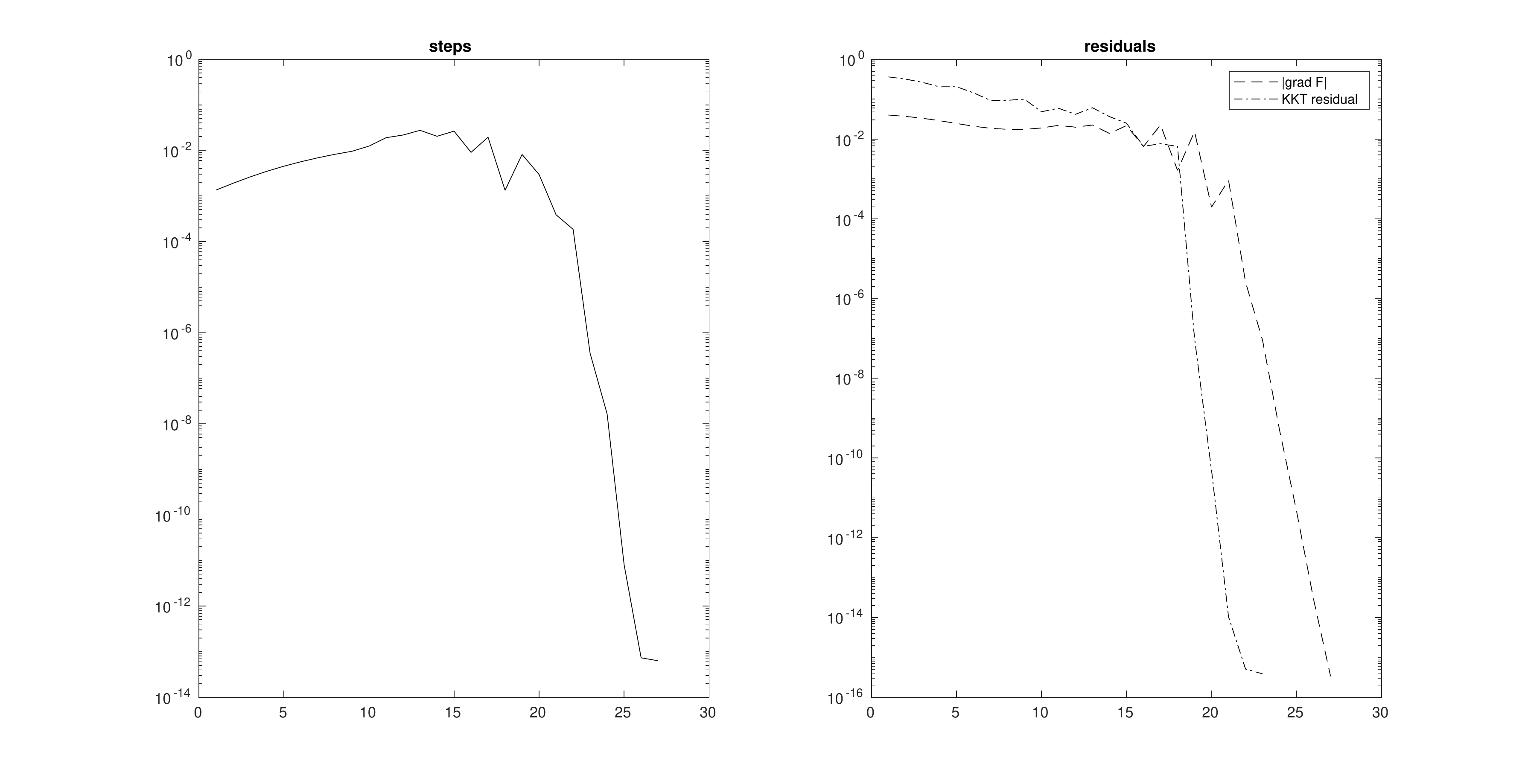}
\end{center}
\end{figure}
\end{landscape}

\begin{figure}[h]
\begin{center}
\caption{The support ($321$ large dots) of the optimal design of Experiment \ref{exp4} (restart with doubled degrees) before the compression by Caratheodory Tchakaloff Theorem, and after compression ($15$ stars).}
\label{disksupportb}
\includegraphics[scale=0.5]{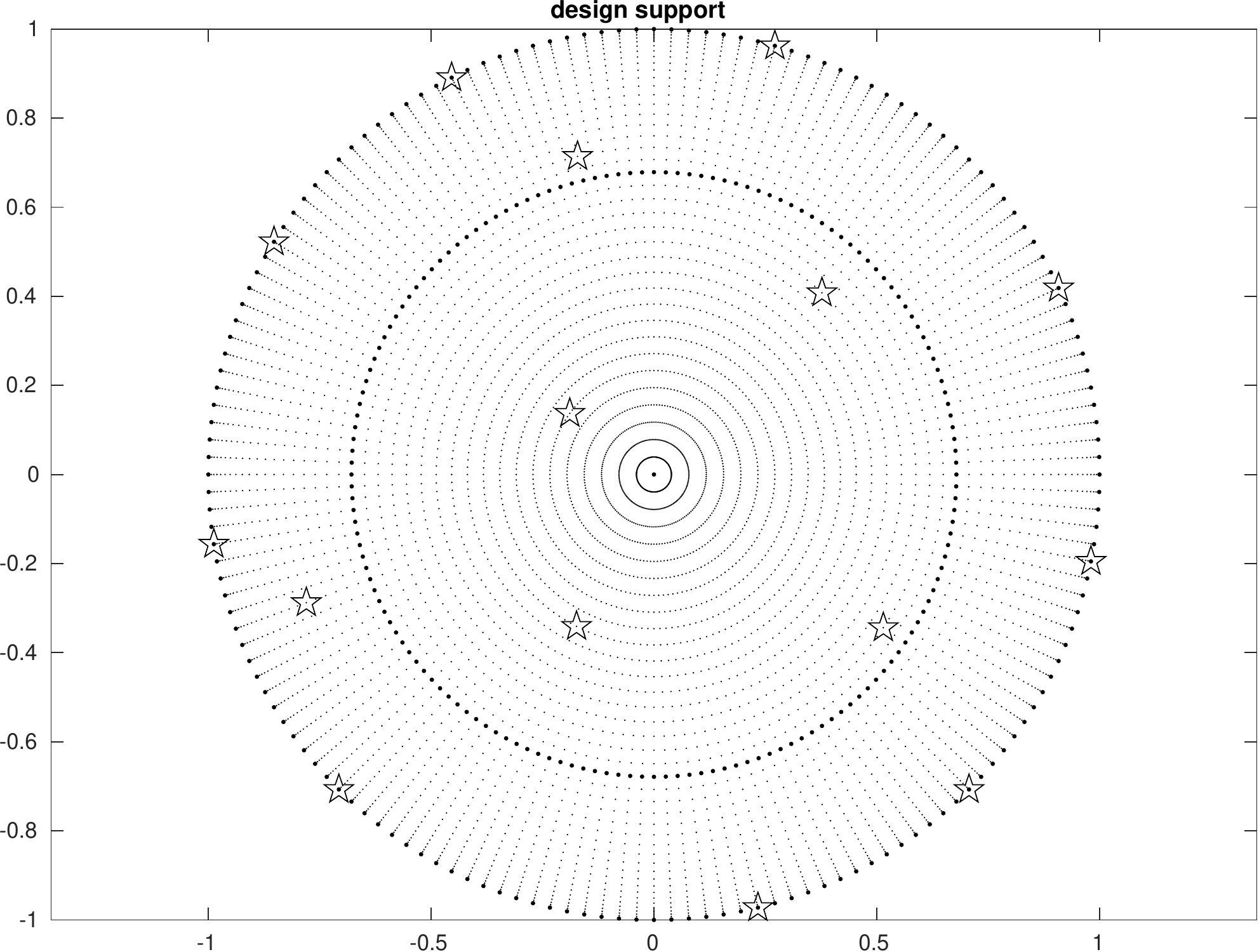}
\end{center}
\end{figure}
As last numerical test, we compare the Karush Kuhn Tucker residual obtained with Algorithm \ref{algo2} and with the Titterington multiplicative algorithm. 
\begin{experiment}[Gaussian random points cloud]\label{exp5}
Let $X$ be a Gaussian random points cloud of size $10000$ in $\reals^2$. Let $\Phi$ the space of polynomials of degree at most $3$ in two real variables. Let $z^0=1/M(1,1,\dots,1)^t)$. Run Algorithm \ref{algo2} with the following parameters setting: $\alpha=\beta=1.15$, $\tau^0=1$, $\epsilon=10^{-4}$, and $r_{max}=5$.
Consider also the Titterington multiplicative algorithm on the same design space and starting at the same initial design $z^0$.
\end{experiment} 
The single iteration (e.g., time step) of Algorithm \ref{algo2} is more computationally expensive with respect to the iteration of the Titterington algorithm (see \ref{SiTiTo78}). Thus we compare the Karush Kuhn Tucker residuals of the two methods at the same CPU time. We report the obtained results in Figure \ref{comparison}. Note that the effect of the different rates of convergence of the two considered algorithms is evident. 
\begin{figure}[h]
\begin{center}
\caption{Results of Experiment \ref{exp5} (10000 random Gaussian Points). The KKT residuals obtained by Algorithm \ref{algo2} and by the Titterington multiplicative algorithm \cite{SiTiTo78} vs CPU time are reported.}
\label{comparison}
\includegraphics[scale=0.45]{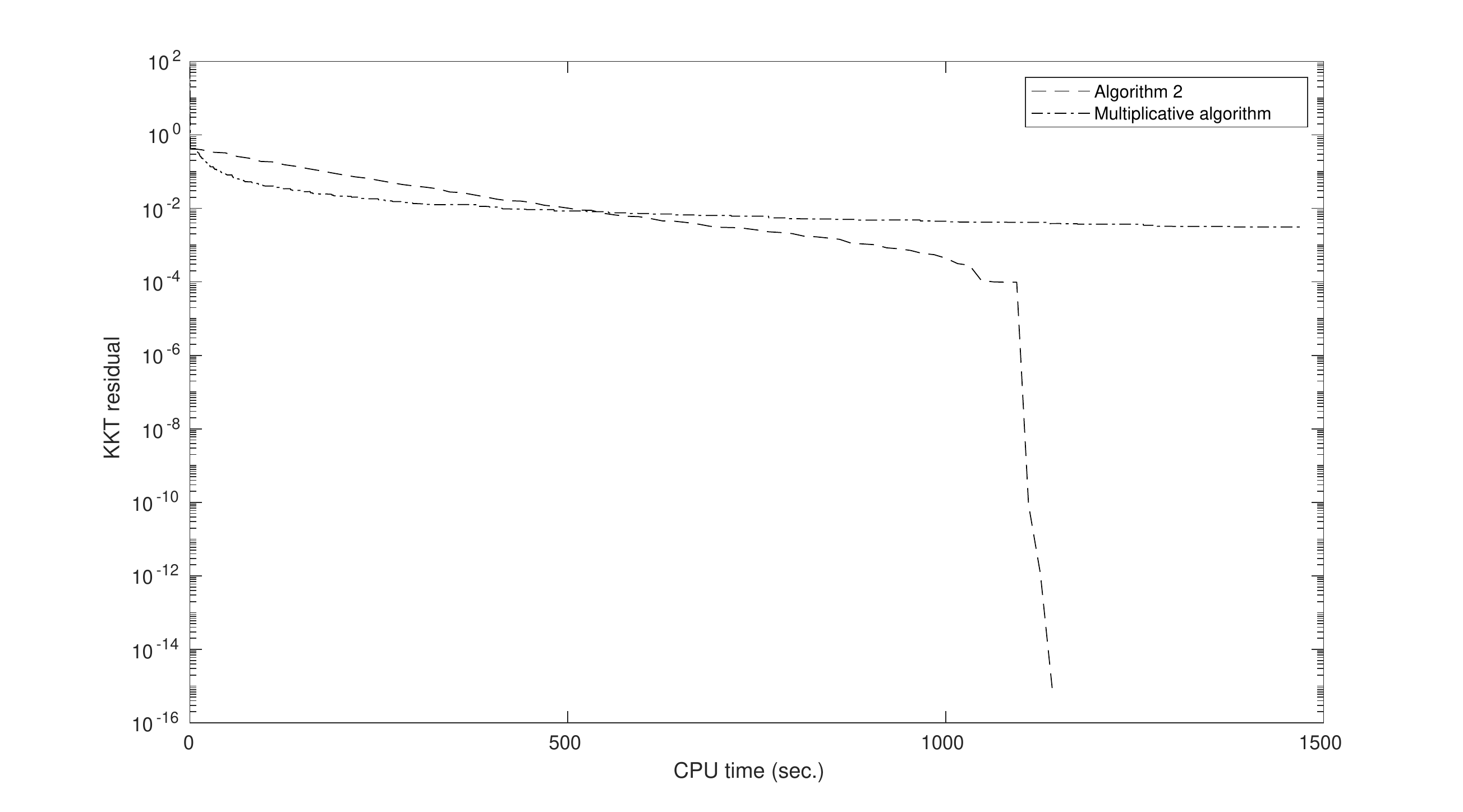}
\end{center}
\end{figure}

\vfill
\bibliographystyle{abbrv}
\bibliography{references}
\end{document}